\documentclass[final]{siamltex}

\usepackage{latexsym,amsmath,amssymb,amsfonts,amscd,multirow,dsfont}
\usepackage{epsfig,verbatim,epstopdf,graphics}
\usepackage{bm,color}
\usepackage{subfigure}
\usepackage{mathtools}

\mathtoolsset{showonlyrefs}

\newtheorem{thm}{Theorem}[section]

\newtheorem{lem}[thm]{Lemma}

\newtheorem{rem}[thm]{Remark}

\newfont{\iams}{msbm9}

\newcommand{\commentbis}[1]{}
\newcommand{\be}{\begin{eqnarray}}
	\newcommand{\ee}{\end{eqnarray}}
\newcommand{\beno}{\begin{eqnarray*}}
	\newcommand{\eeno}{\end{eqnarray*}}

\newcommand{\barr}[1]{\begin{array}{#1}}
	\newcommand{\earr}{\end{array}}

\newcommand{\beq}{\begin{equation}}
	\newcommand{\eeq}{\end{equation}}
\newcommand{\beqa}{\begin{eqnarray}}
	\newcommand{\eeqa}{\end{eqnarray}}

\newcommand{\mD}{{\mathcal{D}}}

\title
{
	Kernel Based High Order ``Explicit'' Unconditionally Stable Scheme for Nonlinear Degenerate Advection-Diffusion Equations
}
\author{
Andrew Christlieb
 \thanks{Department of Computational Mathematics, Scienc and Engineering, Department of Mathematics  and  Department of Electrical Engineering, Michigan State University, East Lansing, MI, 48824. {\tt christli@.msu.edu}}. Research is supported in part by AFOSR grants
 FA9550-12-1-0343, FA9550-12-1-0455, and FA9550-15-1-0282, and NSF grant DMS-1418804.
 \and
	Wei Guo \thanks{
	Department of Mathematics and Statistics, Texas Tech University, Lubbock, TX,  79409. {\tt
		weimath.guo@ttu.edu}. Research
		is supported in part by NSF grant NSF-DMS-1620047.
	}
 \and
	Yan Jiang\thanks{Department of Mathematics, Michigan State University, East Lansing, MI, 48824. {\tt jiangyan@math.msu.edu}}.
}

\date{\today}

\begin{document}

\maketitle
	
\begin{abstract}
In this paper, we present a novel numerical scheme for solving a class of nonlinear degenerate parabolic equations with non-smooth solutions. The proposed method relies on a special kernel based formulation of the solutions found in our early work on the method of lines transpose and successive convolution. In such a framework, a high order weighted essentially non-oscillatory (WENO) methodology and a nonlinear filter are further employed to avoid spurious oscillations. High order accuracy in time is realized by using the high order explicit strong-stability-preserving (SSP) Runge-Kutta method. Moreover, theoretical investigations of the kernel based formulation combined with an explicit SSP method indicates that the combined scheme is unconditionally stable and up to third order accuracy.  Evaluation of the kernel based approach is done with a fast $\mathcal{O}(N)$ summation algorithm. The new method allows for much larger time step evolution compared with other explicit schemes with the same order accuracy, leading to remarkable computational savings.
\end{abstract}


 

\section{Introduction}

In this paper, we are interested in numerically solving the following nonlinear, possibly degenerate, parabolic equation
\begin{align}
\label{eq:ad}
u_{t}+f(u)_{x}=g(u)_{xx},\quad x\in[a,b],
\end{align}
where $g'(u)\geq0$ and $g'(u)$ can vanish for some values of $u$.
Such an equation arises in a wide range of applications, e.g., collisional transport models in plasmas,  radiative transport, porous medium flow, etc. 
The equation \eqref{eq:ad} has similar properties to hyperbolic conservation laws, including possible existence of discontinuous solutions and sharp fronts, and a finite speed of propagation of wave fronts.  When considering models of the form of equation \eqref{eq:ad}, it is necessary to design a numerical scheme capable of capturing these features. 

A variety of schemes have been developed in the literature, e.g., finite volume schemes \cite{bessemoulin2012finite}, finite difference WENO schemes \cite{liu2011high, hajipour2012high, abedian2013high}, local discontinuous Galerkin methods \cite{zhang2009numerical}, kinetic schemes \cite{aregba2004explicit}, and relaxation schemes \cite{cavalli2007high}, among others. Most of these methods are in the method of lines (MOL) framework, meaning that the spatial variable is first discretized, then the numerical solution is updated in time by coupling a suitable time integrator. The most commonly used time evolution methods are the strong-stability-preserving Runge-Kutta (SSP RK) schemes and SSP multi-step schemes \cite{gottlieb2001strong, shu2002survey, gottlieb2005high}. SSP methods preserve the strong stability in some desired norm of an appropriate spatial discretization in conjunction with the forward Euler time stepping, thus preventing spurious oscillations near spatial discontinuities. However, it is well known that an explicit time discretization  does have a restriction on the time step in order to maintain  stability. For example, for advection problems, maintaining stability usually requires the time step $\Delta t\propto \Delta x$, where $\Delta x$ is the spatial mesh size. Solving diffusion problems with explicit time stepping methods  introduces a more stringent restriction $\Delta t\propto \Delta x^2$ for stability. Hence, using an explicit SSP method for equation \eqref{eq:ad} means that satisfying the stability condition demands  $\Delta t\propto \Delta x^2$.  There are two other approaches one can utilize in these situations including Implicit-Explicit (IMEX) methods and fully implicit methods \cite{ascher1997implicit,Hairer1996}.  Both of these approaches are effective in the sense that they permit larger time steps without stability issue.  Except for the backwards Euler method and the second order fully implicit method proposed by Ketcheson \cite{ketcheson2011step},  this class of methods often requires that $\Delta t\propto \Delta x$ to maintain non-oscillatory numerical solutions near discontinuities or steep gradients that arise from the equation \eqref{eq:ad}.
Meanwhile, this approach needs to invert matrices or nonlinear operators resulting from the spatial discretization  at each time step.
This typically involves the use of some form of iterative solvers, Krylov or multi-grid methods.  
In practice, full matrix inversions for problems in the form of equation \eqref{eq:ad} may become prohibitively complicated and costly, especially when memory is extremely limited.

An alternative approach to solving \eqref{eq:ad} is the method of lines transpose (MOL$^T$), also known as Rothes method or transverse method of lines \cite{salazar2000theoretical, schemann1998adaptive, causley2014method} in the literature. In the MOL$^T$ framework, the discretization is first carried out for the temporal variable, resulting in a boundary value problem (BVP) at  discrete time levels. Then a preferred BVP solver is applied to advance the numerical solution.  A notable advantage of the MOL$^T$ approach is that an implicit method, e.g. the backward Euler method, can be used in the first step, and then in the second step, the operator of the BVP, e.g, the modified Helmholtz operator for solving wave equations, is inverted analytically using an integral formulation based on a Green's/kernel function. Because the method utilizes direct inversion of the operator, the MOL$^T$ approach eliminates the need to solve linear systems at each time step.
Moreover, many well-established  fast convolution algorithms can be readily used in reducing the computational complexity of the scheme from $\mathcal{O}(N^2)$ to $\mathcal{O}(N)$ with $N$ being the number of discrete mesh points  \cite{causley2013method,greengard1987fast,barnes1986hierarchical}. 
In \cite{causley2014higher,causley2016method}, a novel technique known as successive convolution (or resolvent expansions) is developed and analyzed for the wave equation and the Allen-Cahn equation. The resulting scheme is unconditionally stable, but would be rarely applied on general nonlinear problems possibly due to the complex formulation. Additionally, the MOL$^T$ framework has been studied to solve the linear and non-linear heat equation \cite{causley2017method, causley2016method, kropinski2011fast, jia2008krylov}, Maxwell's equations\cite{cheng2017asymptotic} as well as others. 
Recently, an implicit high order SSP RK method and a robust WENO integral formulation have been incorporated into the MOL$^T$ framework  so that the method can be applied to the advection equation and the Vlasov equation \cite{christlieb2016weno}. With the help of the SSP property and the WENO based quadrature, the method is able to take large time steps, and at the same time, generate a solution being free of oscillations. Unfortunately this scheme as designed does not directly extend to problems in the form of equation  \eqref{eq:ad}.

In this paper, we will propose a novel numerical scheme  following the MOL$^T$ philosophy of employing a kernel based approach for degenerate parabolic equations of the form \eqref{eq:ad}. The major distinction between the newly proposed scheme and our previous work \cite{christlieb2016weno}, is that we employ explicit time stepping methods that are traditionally used in the MOL formulation.  In particular, we start with  explicit SSP RK methods for the time discretization, which can be easily applied to nonlinear problems.   Following the idea of solving the BVP as in the MOL$^T$ framework, we  transform the spatial derivatives into a kernel based representation.  This approach makes the method effectively ``implicit" at each stage of the explicit SSP RK method, but without the need to invert any matrices. The robust WENO methodology in conjunction with a new nonlinear filter  is adopted with the aim to effectively capture sharp gradients of the solution without producing spurious oscillations. In addition, a special parameter $\beta$ is introduced in the scheme formulation and we are able to make this scheme  A-stable through a careful choice of $\beta$. The unconditional stability property of the scheme can be established accordingly. In summary,  the proposed  scheme we have designed for solving equation \eqref{eq:ad} is robust, high order accurate, matrix free,  unconditionally stable, and efficient.

The paper is organized as follows. In Section 2, we represent the spatial derivatives as infinite series, in which each term relies on a special kernel based formulation of the solution. In Section 3, the approximation accuracy of the associated partial sum is studied. We introduce the WENO-based quadrature as well as the nonlinear filter for evaluation of the partial sum in Section 4. The fully-discrete scheme for solving \eqref{eq:ad} is designed by coupling the partial sum formulation with the high order explicit SSP RK method, and a stability analysis for linear problems is established in Section 5.  This method can be extended directly to high dimensional problems, and we discuss the details of the two-dimensional formulation in Section 6. We present several numerical tests in Section 7 to verify the performance of the proposed scheme. Finally, we conclude with a brief discussion in Section 8. 

\section{Representation of differential operators}
In this section, we will review a class of representations of the first spatial derivative $\partial_{x}$ and the second spatial derivative $\partial_{xx}$ from the nonlinear convection-diffusion equation \eqref{eq:ad}. Such representations are based on a successive convolution of the underlying kernel functions and will serve as the key building block of the proposed scheme.  Below, we first introduce the operator $\mathcal{L}$ and the associated operator  $\mathcal{D}$. Then, both differential operators $\partial_x$ and $\partial_{xx}$ are represented by infinite series of $\mathcal{D}$. 
Such an idea was first used in \cite{causley2014higher} for designing an MOL$^T$ A-stable scheme for solving linear wave equations. 

\subsection{The second order derivative $\partial_{xx}$}
\label{sec:der_xx}
We consider the following differential operator: 
\begin{align}
\label{eq:operL0}
\mathcal{L}_{0}=\mathcal{I}-\frac{1}{\alpha^2}\partial_{xx},  \quad x \in[a,b],
\end{align}
where $\mathcal{I}$ is the identity operator and $\alpha>0$ is a constant. Suppose $w(x)$ satisfies the differential equation
\begin{align}
\label{eq:molt0}
\mathcal{L}_{0}[w,\alpha](x)=w(x)-\frac{1}{\alpha^2}w_{xx}(x)=v(x),
\end{align}
where $v(x)$ is a given function.
Then, by analytically inverting operator $\mathcal{L}_0$, we can obtain the explicit expression of $w(x)$ as
\begin{align}
\label{eq:L0inverse}
w(x)=\mathcal{L}_{0}^{-1}[v,\alpha](x)=
I^{0}[v,\alpha](x) + A_{0}e^{-\alpha (x-a)} + B_{0}e^{- \alpha (b-x)},
\end{align}
where
\begin{align}
\label{eq:I0}
I^{0}[v,\alpha](x):=\frac{\alpha}{2} \int_a^b e^{-\alpha |x-y|}v(y)dy,
\end{align} 
and $A_{0}$ and $B_{0}$ are constants determined by the boundary conditions, see \cite{causley2014method}. For example, for periodic boundary conditions, i.e., $w(a)=w(b)$ and $w_{x}(a)=w_{x}(b)$, we have
\begin{align}
\label{eq:bcD0per}
A_{0}=\frac{I^{0}[v,\alpha](b)}{1-\mu}\ \ \ \text{and} \ \ \ B_{0}=\frac{I^{0}[v,\alpha](a)}{1-\mu},
\end{align}
with $\mu=e^{-\alpha(b-a)}$. 

We then define the operator $\mathcal{D}_0$ as 
\begin{align}
\label{eq:d0}
\mathcal{D}_{0}=\mathcal{I}-\mathcal{L}_{0}^{-1}.
\end{align} 
Clearly, $\mathcal{L}_{0}=(\mathcal{I}-\mathcal{D}_{0})^{-1}$. Moreover, by the definition \eqref{eq:operL0}, the second derivative can be rewritten as 
\begin{align}
\label{eq:partial2}
\frac{1}{\alpha^2}\partial_{xx}=\mathcal{I}-\mathcal{L}_{0}
=\mathcal{L}_{0}(\mathcal{L}_{0}^{-1}-\mathcal{I}) 
=-\mathcal{D}_{0}(\mathcal{I}-\mathcal{D}_{0})^{-1}
=-\sum_{p=1}^{\infty}\mathcal{D}_{0}^{p},
\end{align}
where $\mathcal{D}_{0}^{p}$ is successively defined as $\mathcal{D}_{0}^{p}=\mathcal{D}_{0}[\mathcal{D}_{0}^{p-1}]$. 
Hence, $g(u)_{xx}$ from \eqref{eq:ad} can be represented as 
\begin{align}
\label{eq:g_xx}
g(u)_{xx} = -\alpha^2 \sum_{p=1}^{\infty} \mathcal{D}^{p}_{0}[g(u),\alpha](x).
\end{align}

\subsection{The first derivative $\partial_{x}$}
Following a similar idea, we are able to deal with the first derivative $\partial_x$ as well. Note that, when designing a numerical method for solving a hyperbolic conservation law \begin{equation}
u_{t}+f(u)_x=0,\label{eq:cl}
\end{equation} we have to consider the propagation direction of the wave solution to ensure stability of the scheme. To this end, we introduce operators $\mathcal{L}_{L}$ and $\mathcal{L}_{R}$ and the associated operators $\mathcal{D}_{L}$ and $\mathcal{D}_{R}$ to account for waves traveling in opposite directions:
\begin{align}
\label{eq:operLL}
\mathcal{L}_{L}&=\mathcal{I}+\frac{1}{\alpha}\partial_{x}, \ \ \
\mathcal{D}_{L}=\mathcal{I}-\mathcal{L}^{-1}_{L};\\
\label{eq:operLR}
\mathcal{L}_{R}&=\mathcal{I}-\frac{1}{\alpha}\partial_{x}, \ \ \
\mathcal{D}_{R}=\mathcal{I}-\mathcal{L}^{-1}_{R},
\end{align} 
where $x\in[a,b]$ and $\alpha>0$ is a constant. 

Let us consider $\mathcal{L}_{L}$ and $\mathcal{D}_{L}$ first. Assume $w(x)$ satisfies the differential equation 
$$\mathcal{L}_L[w,\alpha](x) = w(x)+\frac{1}{\alpha}\partial_xw(x)= v(x), $$
where $v(x)$ is a given function. 
As with $\mathcal{L}_0$, we can analytically invert $\mathcal{L}_{L}$ as follows:
\begin{align}
\label{eq:LLinverse}
\mathcal{L}_{L}^{-1}[v,\alpha](x)=I^{L}[v,\alpha](x) + A_{L}e^{-\alpha (x-a)},
\end{align}
where
\begin{align}
\label{eq:IL}
I^{L}[v,\alpha](x)=\alpha \int_a^x e^{-\alpha (x-y)}v(y)dy,
\end{align}
and the constant $A_L$ is determined by the boundary condition. For example, for periodic boundary conditions \eqref{eq:operLL},
$A_{L}$ is taken as 
\begin{align}
\label{eq:bcDLper}
A_L=\frac{I^{L}[v,\alpha](b)}{1-\mu}.
\end{align}
Similar to $\frac{1}{\alpha^2}\partial_{xx}$, we are able to represent $\frac1\alpha\partial_{x}$ using an infinite series of the operator $\mathcal{D}_L$
\begin{equation}
\label{eq:ll}
\frac{1}{\alpha}\partial_{x}=\mathcal{L}_{L}-\mathcal{I}
=\sum_{p=1}^{\infty}\mathcal{D}_{L}^{p}.
\end{equation}
Note that, $I_{L}[v,\alpha](x)$ only depends on the function values of $v$ from the left end point $a$ to $x$. On the other hand, it is well-known that, for the hyperbolic conservation law \eqref{eq:cl}, the information of the solution propagates from left to right over time if the flux function $f(u)$ has a positive derivative, i.e.,  $f'(u)\geq0$. Hence, it is reasonable to represent $f(u)_x$ as
\begin{align}
\label{eq:partialL}
\partial_{x}f(u)=
\alpha\sum_{p=1}^{\infty}\mathcal{D}_{L}^{p}[f(u),\alpha](x).
\end{align}

Similarly,  when $f'(u)\leq0$, we can represent $f(u)_x$ using the following infinite series of $\mathcal{D}_R$:
\begin{align}
\label{eq:partialR}
\partial_{x}f(u)
=-\alpha\sum_{p=1}^{\infty}\mathcal{D}_{R}^{p}[f(u),\alpha](x),
\end{align}
where $\mathcal{D}_R$ is given in \eqref{eq:operLR} with
\begin{align}
\label{eq:LRinverse}
& \mathcal{L}_{R}^{-1}[v,\alpha](x)=I^{R}[v,\alpha](x) + B_{R} e^{-\alpha (b-x)},\\
\label{eq:IR}
& I^{R}[v,\alpha](x)=\alpha \int_x^{b} e^{-\alpha (y-x)}v(y)dy,
\end{align}
and $B_R$ is a constant determined by the boundary condition. 
For example, for  periodic boundary conditions,
\begin{align}
\label{eq:bcDRper}
B_R=\frac{I^{R}[v,\alpha](a)}{1-\mu}.
\end{align}


In the case of $f(u)$ not being a monotone function of $u$, we can employ the following global ``flux splitting" strategy:
\begin{equation}\label{eq:fluxsplitting}
f(u)=\frac{1}{2}(f^{+}(u)+f^{-}(u)),
\end{equation}
with $df^{+}(u)/du\geq0$ and $df^{-}(u)/du\leq0$. We then use \eqref{eq:partialL} to represent $f^{+}(u)_x$, and \eqref{eq:partialR} to represent $f^{-}(u)_x$. The most commonly used splitting strategy is the Lax-Friedrich splitting
$$f^{\pm}(u)=\frac{1}{2}(f(u)\pm cu), \ \ \ \text{with} \ \ \ c=\max_{u}|f'(u)|,$$
which has been widely used in design of high order finite difference schemes for conservation laws \cite{shu1988efficient}.

In summary,  $-f(u)_{x}+g(u)_{xx}$ can be represented as a linear combination of three infinite series:
\begin{align}
	\label{eq:expall}
	-\alpha_{L}\sum_{p=1}^{\infty}\mathcal{D}_{L}^{p}[f^{+}(u),\alpha_{L}]
	+\alpha_{R}\sum_{p=1}^{\infty}\mathcal{D}_{R}^{p}[f^{-}(u),\alpha_{R}]
	-\alpha_{0}^2 \sum_{p=1}^{\infty} \mathcal{D}^{p}_{0}[g(u),\alpha_{0}].
\end{align}

\section{Approximation of partial sums}
So far, we have shown that the derivatives $\partial_{x}$ and $\partial_{xx}$  can be represented as infinite series. In numerical simulations, we have to truncate the series and only compute the corresponding partial sum.  In particular, \eqref{eq:expall} is approximated by the $k^{th}$ partial sum
\begin{align}
\label{eq:partial}
-f(u)_{x}+g(u)_{xx}\approx-\alpha_{L}\sum_{p=1}^{k}\mathcal{D}_{L}^{p}[f^{+}(u),\alpha_{L}]
+\alpha_{R}\sum_{p=1}^{k}\mathcal{D}_{R}^{p}[f^{-}(u),\alpha_{R}]
-\alpha_{0}^2 \sum_{p=1}^{k} \mathcal{D}^{p}_{0}[g(u),\alpha_{0}].
\end{align}
In this section, we will theoretically investigate the truncation error incurred. Below, we restrict our attention to periodic boundary conditions
\begin{align}
\label{eq:bcper}
u(a,t)=u(b,t), \ \ \ u_{x}(a,t)=u_{x}(b,t), \ \ \ t\geq0,
\end{align}
and the following special homogeneous boundary condition
\begin{align}
\label{eq:bcdir}
\partial^{p}_{x}u(a,t)=0, \ \ \ \partial^{p}_{x}u(b,t)=0, \ \ \ t\geq0, \ \ \ p\geq1.
\end{align} 
Note that a solution that is constant near the boundary satisfies such a homogeneous boundary condition.
The investigation of other boundary conditions is left to the future work.

\subsection{Periodic boundary conditions}
 
In the case of periodic boundary conditions \eqref{eq:bcper}, it is straightforward to require
\begin{align}
\label{eq:bcDper1}
 \mathcal{D}^{p}_{0}[g(u),\alpha_{0}](a) &= \mathcal{D}^{p}_{0}[g(u),\alpha_{0}](b), \notag\\
 \mathcal{D}^{p}_{L}[f^{+}(u),\alpha_{L}](a) &= \mathcal{D}^{p}_{L}[f^{+}(u),\alpha_{L}](b),\\ 
 \mathcal{D}^{p}_{R}[f^{-}(u),\alpha_{R}](a) &= \mathcal{D}^{p}_{R}[f^{-}(u),\alpha_{R}](b),\notag
\end{align}
for $p=1,\,2,\,3,\ldots,\, k$. Consequently, we can obtain the coefficients $A_0$, $B_0$, $A_L$, and $B_R$ by \eqref{eq:bcD0per}, \eqref{eq:bcDLper} and \eqref{eq:bcDRper}. 

With the boundary treatments \eqref{eq:bcDper1}, we are able to establish the following theorem, which provides error estimates when the infinite series \eqref{eq:g_xx},  \eqref{eq:partialL} and \eqref{eq:partialR} are truncated by the corresponding $k^{th}$ partial sum.  
\begin{thm}
	\label{thm1}
	Suppose $v(x)$ is a periodic smooth function.
	\begin{enumerate}
		\item Consider the operator $\mathcal{D}_{0}$ with the boundary treatment $\mathcal{D}_{0}(a)=\mathcal{D}_{0}(b)$, If $v(x)\in\mathit{C}^{2k+2}[a,b]$, then we have 
		\begin{align}
		\| \partial_{xx}v(x) + \alpha^2\sum_{p=1}^{k}\mathcal{D}_{0}^{p}[v,\alpha](x) \|_{\infty} \leq C \left(\frac{1}{\alpha}\right)^{2k} \|\partial^{2k+2}_{x}v(x)\|_{\infty}
		\end{align}
		where $C$ is a constant only depending on $k$.
		
		\item Consider the operator  $\mathcal{D}_{L}$ and $\mathcal{D}_{R}$ with the boundary treatment $\mathcal{D}_{L}(a)=\mathcal{D}_{L}(b)$ and $\mathcal{D}_{R}(a)=\mathcal{D}_{R}(b)$, respectively.
		 If $v(x)\in\mathit{C}^{k+1}[a,b]$, then we have 
		\begin{align}
		\| \partial_{x}v(x) - \alpha \sum_{p=1}^{k}\mathcal{D}_{L}^{p}[v,\alpha](x) \|_{\infty} \leq C
		\left(\frac{1}{\alpha}\right)^{k}  \|\partial^{k+1}_{x}v\|_{\infty},
		\end{align}
		and
		\begin{align}
		\| \partial_{x}v(x) + \alpha\sum_{p=1}^{k}\mathcal{D}_{R}^{p}[v,\alpha](x) \|_{\infty} \leq C
		\left(\frac{1}{\alpha}\right)^{k} \|\partial^{k+1}_{x}v\|_{\infty},
		\end{align}
		where $C$ is a constant depending only on $k$.
	
	\end{enumerate}
\end{thm}

To prove this theorem, we first introduce the following lemma regarding operator $\mathcal{D}_{*}$, where $*$ can be $0$, $L$ and $R$.
\begin{lem}
	\label{lem1} Suppose $v(x)$ is a periodic smooth function.
	\begin{enumerate}
		\item For the operator $\mathcal{D}_{0}$ with the boundary treatment $\mathcal{D}_{0}(a)=\mathcal{D}_{0}(b)$,  we have 
		\begin{align}
		\mathcal{D}_{0}[v,\alpha](x)
		=&  -\sum_{p=1}^{k}\left(\frac{1}{\alpha}\right)^{2p} \partial^{2p}_{x}v(x)  -\left(\frac{1}{\alpha}\right)^{2k+2}\mathcal{L}^{-1}_{0}[\partial^{2k+2}_{x}v,\alpha](x).
		\end{align}
		if $v(x)\in\mathit{C}^{2k+2}[a,b]$.
		
		\item For the operators $\mathcal{D}_{L}$ and $\mathcal{D}_{R}$ with the boundary treatment $\mathcal{D}_{L}(a)=\mathcal{D}_{L}(b)$ and $\mathcal{D}_{R}(a)=\mathcal{D}_{R}(b)$, respectively, we have 
		\begin{subequations}
			\begin{align}
			\mathcal{D}_{L}[v,\alpha](x)
			=&  -\sum_{p=1}^{k}\left(-\frac{1}{\alpha}\right)^{p} \partial^{p}_{x}v(x)   
			+ \left(-\frac{1}{\alpha}\right)^{k+1} \mathcal{L}^{-1}_{L}[\partial^{k+1}_{x}v,\alpha](x),\\
			\mathcal{D}_{R}[v,\alpha](x)
			=&  -\sum_{p=1}^{k}\left(\frac{1}{\alpha}\right)^{p} \partial^{p}_{x}v(x)   
			- \left(\frac{1}{\alpha}\right)^{k+1} \mathcal{L}^{-1}_{R}[\partial^{k+1}_{x}v,\alpha](x),
			\end{align}
		\end{subequations}
		if $v(x)\in\mathit{C}^{k+1}[a,b]$.
	\end{enumerate} 
\end{lem}

The proof of the lemma heavily relies on integration by parts, and it is provided in the appendix. Below, we provide the proof of Theorem \ref{thm1}.

\begin{proof}
		For brevity, we only show the details of the proof for case 1. Following a similar argument, one can easily prove the case 2. 
		
		First, by repeating the proof of Lemma \ref{lem1} with $\mathcal{D}_{0}(a)=\mathcal{D}_0(b)$, we have the following equality for any $m$ with $0\leq m<k$:
		\begin{align*}
		\mathcal{D}_{0}[\partial_{x}^{2m}v,\alpha](x)= -\sum_{p=m+1}^{k} \left( \frac{1}{\alpha} \right)^{2(p-m)}\partial^{2p}_{x}v(x) - \left( \frac{1}{\alpha} \right)^{2(k+1-m)} \mathcal{L}^{-1}_{0}[\partial^{2k+2}_{x}v,\alpha](x).
		\end{align*}
		Meanwhile, it is easy to verify that $\partial_{xx}\mathcal{L}_{0}^{-1}[v,\alpha](x) = \alpha^2\mathcal{L}_{0}^{-1}[v,\alpha](x)-\alpha^2 v(x) =\alpha^2\mathcal{D}_{0}[v,\alpha](x) $. 
		Therefore, we have
		$\partial_{xx}\mathcal{D}_{0}[v,\alpha](x) =\partial_{xx}v(x)+\alpha^2\mathcal{D}_{0}[v,\alpha](x)$.
		Furthermore, for $0\leq p\leq k$, there is a general form 
		\begin{align*}
		\partial^{2p}_{x}\mathcal{D}_{0}[v,\alpha](x)
		=& \sum_{m=1}^{p}\alpha^{2(p-m)}\partial^{2m}_{x}v(x) + \alpha^{2p}\mathcal{D}_{0}[v,\alpha](x)\\
		=& - \sum_{m=p+1}^{k}\alpha^{2(m-p)}\partial^{2(m-p)}_{x}v(x) - \left( \frac{1}{\alpha} \right)^{2(k+1-p)} \mathcal{L}^{-1}_{0}[\partial^{2k+2}_{x}v,\alpha](x). 
		\end{align*}
		Next, let us consider the operator $\mathcal{D}_{0}^2=\mathcal{D}_{0}[\mathcal{D}_{0}]$. Since the periodic boundary treatment \eqref{eq:bcDper1} is imposed for $\mathcal{D}_{0}$, we deduce that
		\begin{small}
		\begin{align*}
		\mathcal{D}_{0}^{2}[v,\alpha](x) 
		=& -\sum_{p=1}^{k-1} \left( \frac{1}{\alpha} \right)^{2p}\partial^{2p}_{x}\mathcal{D}_{0}[v,\alpha](x) - \left( \frac{1}{\alpha} \right)^{2k} \mathcal{L}^{-1}_{0}[\partial^{2k}_{x}\mathcal{D}_{0}[v,\alpha],\alpha](x)\\
		=& -\sum_{p=1}^{k-1} \left(\frac{1}{\alpha} \right)^{2p} \left( -\sum_{m=p+1}^{k} \left( \frac{1}{\alpha} \right)^{2(m-p)}\partial^{2m}_{x}v(x) - \left( \frac{1}{\alpha} \right)^{2(k+1-p)} \mathcal{L}^{-1}_{0}[\partial^{2k+2}_{x}v,\alpha](x) \right)\\
		& - \left( \frac{1}{\alpha} \right)^{2k} \mathcal{L}^{-1}_{0}[-\frac{1}{\alpha^2}\mathcal{L}^{-1}_{0}[\partial^{2k+2}_{x}v,\alpha],\alpha](x)\\
		=& \sum_{p=2}^{k}(p-1)\left(\frac{1}{\alpha}\right)^{2p} \partial^{2p}_{x}v(x) + (k-1) \left(\frac{1}{\alpha}\right)^{2k+2} \mathcal{L}^{-1}_{0}[\partial^{2k+2}_{x}v,\alpha](x)
		+ \left( \frac{1}{\alpha} \right)^{2k+2} \left(\mathcal{L}_{0}^{-1}\right)^2 [\partial^{2k+2}_{x}v,\alpha](x).
		\end{align*}
		\end{small}
		Therefore, 
		\begin{align*}
		\mathcal{D}_{0}[v,\alpha](x)+\mathcal{D}_{0}^{2}[v,\alpha](x)=-\frac{1}{\alpha^2}\partial_{xx}v(x) + \sum_{p=3}^{k}(p-2)\left(\frac{1}{\alpha}\right)^{2p} \partial^{2p}_{x}v(x) + \left(\frac{1}{\alpha}\right)^{2k+2}Q_{2}(x),
		\end{align*}
		where $Q_{2}(x)=(k-2)\mathcal{L}^{-1}_{0}[\partial^{2k+2}_{x}v,\alpha](x) + \left(\mathcal{L}^{-1}_{0}\right)^2[\partial^{2k+2}_{x}v,\alpha](x)$. Repeating the process, and finally, we arrive at
		$$\sum_{p=1}^{k}\mathcal{D}^{p}_{0}[v,\alpha](x)=-\frac{1}{\alpha^2}\partial_{xx}v(x) + \left(\frac{1}{\alpha}\right)^{2k+2} Q_{k}(x),$$
		where $Q_{k}(x)$ is a linear combination of functions $\left(\mathcal{L}^{-1}_{0}\right)^{p}[\partial^{2k+2}_{x}v,\alpha](x)$, $p=1,2,\ldots,k$. Note that, for any $w(x)\in\mathit{C}[a,b]$, we have
		$$\|\mathcal{L}^{-1}_{0}[w,\alpha](x)\|_{\infty}\leq C_{0} \|w\|_{\infty},$$
		where $C_{0}$ is a constant independent of $w$ and $\alpha$. Then, there is a constant $C$ only depending on $k$, such that
		$$\|\partial_{xx}v(x)+\alpha^2\sum_{p=1}^{k}\mathcal{D}^{p}_{0}[v,\alpha](x)\|_{\infty}=\|\left(\frac{1}{\alpha}\right)^{2k}Q_{k}(x)\|_{\infty}\leq C\left(\frac{1}{\alpha}\right)^{2k} \|\partial_{x}^{2k+2}v(x)\|_{\infty},$$
		which completes the proof.
\end{proof}

\begin{rem}
	For the numerical schemes formulated below, we will choose 
	\begin{subequations}
		\begin{align}
		\alpha_{0}&=\sqrt{\frac{\beta}{b\Delta t}}, \ \ \ b=\max_{u}|g'(u)|,\\
		\alpha_{L}&=\alpha_{R}=\frac{\beta}{c\Delta t}, \ \ \ c=\max_{u}|f'(u)|,
		\end{align}
	\end{subequations}
	in \eqref{eq:partial}.
	Here, $\Delta t$ denotes the time step and $\beta$ is a prescribed constant independent of $\Delta t$. Define
	\begin{align}
	\label{eq:H}
	\mathcal{H}[u](x)=& -\frac{\beta}{c\Delta t} \sum_{p=1}^{k}\mathcal{D}_{L}^{p}[f^{+}(u),\frac{\beta}{c\Delta t}](x) 
	+ \frac{\beta}{c\Delta t} \sum_{p=1}^{k}\mathcal{D}_{R}^{p}[f^{-}(u),\frac{\beta}{c\Delta t}](x) \nonumber\\
	& -\frac{\beta}{b\Delta t}\sum_{p=1}^{k}\mathcal{D}_{0}^{p} [g(u),\sqrt{\frac{\beta}{b\Delta t}}](x),	\end{align}
	which will approximate $-f(u)_x+g(u)_{xx}$ with accuracy $\mathcal{O}(\Delta t^{k})$.
\end{rem}

\subsection{Homogeneous boundary condition}
With the boundary condition $\partial_{x}^{p}u(a)=0$ and $\partial_{x}^{p}u(b)=0$, $p\geq1$, we require
\begin{align}
\label{eq:bcDhomo}
& \mathcal{D}^{p}_{0}[g(u),\alpha_{0}](a)=0,\quad \mathcal{D}^{p}_{0}[g(u),\alpha_{0}](b)=0, \notag\\
& \alpha_{L}\mathcal{D}^{p}_{L}[f^{+}(u),\alpha_{L}](a) - \alpha_{R}\mathcal{D}^{p}_{R}[f^{-}(u),\alpha_{R}](a)=0,\\ 
& \alpha_{L}\mathcal{D}^{p}_{L}[f^{+}(u),\alpha_{L}](b) - \alpha_{R}\mathcal{D}^{p}_{R}[f^{-}(u),\alpha_{R}](b)=0,\notag
\end{align}
for $p=1,\,2,\,3,\ldots,\, k$. The coefficients are obtained from the following formula:
\begin{itemize}
\item The operator $\mathcal{D}_{0}$ is required to satisfy
	\begin{align}
	\label{eq:gxxhomo}
	\mathcal{D}_{0}[v,\alpha](a)=\mathcal{D}_{0}[v,\alpha](b)=0
	\end{align}
	with a given function $v(x)$. Then, we have
	\begin{align}
	& A_{0}=\frac{ \mu\left(I^{0}[v,\alpha](b)-v(b)\right) - \left(I^{0}[v,\alpha](a)-v(a)\right)}{1-\mu^2}, \\
	& B_{0}=\frac{ \mu\left(I^{0}[v,\alpha](a)-v(a)\right) - \left(I^{0}[v,\alpha](b)-v(b)\right)}{1-\mu^2}.
	\end{align}
	\item The operators $\mathcal{D}_{L}$ and $\mathcal{D}_{R}$ are required to satisfy
	\begin{align}
	\label{eq:fxhomo}
	\mathcal{D}_{L}[v_{1},\alpha](a)-\mathcal{D}_{R}[v_{2},\alpha](a)=0, \ \ \ \text{and} \ \ \ \mathcal{D}_{L}[v_{1},\alpha](b)-\mathcal{D}_{R}[v_{2},\alpha](b)=0
	\end{align}
 	with given functions $v_{1}(x)$ and $v_{2}(x)$. Then,
	\begin{align}
	& A_{L}=\frac{\mu\left(v_{2}(b)-v_{1}(b)+I^{L}[v_{1},\alpha](b)\right) - \left(v_{2}(a)-v_{1}(a)-I^{R}[v_{2},\alpha](a)\right)}{1-\mu^2}, \\
	& B_{R}= \frac{\left(v_{2}(b)-v_{1}(b)+I^{L}[v_{1},\alpha](b)\right) - \mu \left(v_{2}(a)-v_{1}(a)-I^{R}[v_{2},\alpha](a)\right)}{1-\mu^2}.
	\end{align}
\end{itemize}

By analogy with Theorem \ref{thm1}, we can establish the error estimate for the partial sum \eqref{eq:partial}. We omit the proof, since it is quite similar to that of Theorem \ref{thm1}.

\begin{thm}
	\label{thm2}
	Suppose $v(x)$ is a function defined on $[a,b]$ with the homogeneous boundary condition that $\partial_{x}^{p}v(a)=\partial_{x}^{p}v(b)=0$, for $p\geq1$.
	\begin{enumerate}
		\item Consider the operator $\mathcal{D}_{0}$ with the boundary treatment \eqref{eq:gxxhomo}. If $v(x)\in\mathit{C}^{2k+2}[a,b]$, then we have 
		\begin{align}
		\| \partial_{xx}v(x) + \alpha^2\sum_{p=1}^{k}\mathcal{D}_{0}^{p}[v,\alpha](x) \|_{\infty} \leq C \left(\frac{1}{\alpha}\right)^{2k} \|\partial^{2k+2}_{x}v(x)\|_{\infty},
		\end{align}
		where $C$ is a constant only depending on $k$.
		
		\item Assume $v(x)=v_{1}(x)+v_{2}(x)$ and both $v_1(x)$ and $v_2(x)$ satisfy the homogeneous boundary condition. Consider the operator  $\mathcal{D}_{L}$ and $\mathcal{D}_{R}$ with the boundary treatment \eqref{eq:fxhomo}.
		 If $v_{1}(x),v_{2}(x)\in\mathit{C}^{k+1}[a,b]$, then we have 
		\begin{align}
		\| \partial_{x}v(x) - \left(\alpha \sum_{p=1}^{k}\mathcal{D}_{L}^{p}[v_{1},\alpha](x) - \alpha\sum_{p=1}^{k}\mathcal{D}_{R}^{p}[v_{2},\alpha](x) \right)\|_{\infty} \leq C
		\left(\frac{1}{\alpha}\right)^{k}  \|\partial^{k+1}_{x}v(x)\|_{\infty},
		\end{align}
		where $C$ is a constant depending only on $k$.
	
	\end{enumerate}
\end{thm}

\begin{rem}
For the homogeneous boundary condition case, we can still use \eqref{eq:H} to approximate $-f(u)_x+g(u)_{xx}$ with accuracy $\mathcal{O}(\Delta t^{k})$.
\end{rem}

\section{Space discretization}

In this section, we present the details about the spatial discretization   of $\mathcal{H}[u]$ in \eqref{eq:H}. The proposed algorithm is based on our early work on a high order WENO MOL$^T$  schemes for transport problems. 
Suppose the domain $[a,b]$ is divided by $N+1$ uniformly distributed grid points
$$a=x_0<x_1<\cdots<x_{N-1}<x_N=b,$$
with the mesh size $\Delta x=\frac{b-a}{N}$. Denote $u^n_i$ as the numerical solution at the spatial location $x_{i}$ and time level $t^{n}$. 
On each grid point $x_{i}$, we further denote $I^{*}[v,\alpha](x_{i})$ as $I^{*}_{i}$, where $*$ can be $0$, $L$ and $R$. Note that the convolution integrals $I^{L}_{i}$ and $I^{R}_{i}$ satisfy a recursive relation 
\begin{subequations}
\label{eq:recursive}
\begin{align}
& I^L_i = I^L_{i-1}e^{-\alpha_{L}\Delta x} + J^L_i,\quad i=1,\ldots,N, \quad I^L_0 = 0, \\
& I^R_i = I^R_{i+1}e^{-\alpha_{R}\Delta x} + J^R_i,\quad i=0,\ldots,N-1, \quad I^R_N = 0,
\end{align}
\end{subequations}
respectively, where
\begin{align}
\label{eq:JLR}
 J^L_{i} =  \alpha_{L} \int_{x_{i-1}}^{x_{i}} v(y)e^{-\alpha_{L} (x_{i}-y)}dy,\ \ \ \
 J^R_{i} =  \alpha_{R} \int_{x_{i}}^{x_{i+1}} v(y)e^{-\alpha_{R} (y-x_{i})}dy.
\end{align}
Therefore, once we have computed $J^{L}_{i}$ and $J^{R}_{i}$ for all $i$, we then can obtain $I^{L}_{i}$ and $I^{R}_{i}$ via the recursive relation.
In addition, the convolution integral $I^{0}[v,\alpha_{0}](x)$ can be split into $I^{L}[v,\alpha_{0}](x)$ and $I^{R}[v,\alpha_{0}](x)$,  
$$ I^{0}[v,\alpha_{0}](x) =\frac{1}{2}( I^{L}[v,\alpha_{0}](x) + I^{R}[v,\alpha_{0}](x).$$
Thus, $I^0_i$ can be evaluated in the same way as $I^L_i$ and $I^R_i $, see \cite{causley2013method}. 

A distinct feature of the equation \eqref{eq:ad} is that
discontinuous solution structures and sharp fronts may be developed.  The WENO methodology has long been a standard tool to solve hyperbolic problems with discontinuous solutions, which can achieve sharp and non-oscillatory shock transitions and high order accuracy in smooth regions \cite{jiang1996efficient,shu1998essentially, shu2009high}. 
Recently, in \cite{liu2011high}, the authors developed a finite difference WENO scheme to solve degenerate parabolic equations. Such an approach directly approximates the second derivative term using a conservative flux difference formulation. 
In \cite{christlieb2016weno}, a WENO-based high order quadrature was developed to evaluate $J^{L}_{i}$ and $J^{R}_{i}$. Some related works in the literature about the WENO-based quadrature include \cite{chou2006high, chou2007high, liu2009positivity}. In this work,  we still employ the WENO-based quadrature from \cite{christlieb2016weno} with the aim to avoid spurious oscillations when solving discontinuous problems. For the reader's convenience, we will briefly review the main procedure. All the formulas and the associated coefficients are provided as well.   Note that, as suggested in our numerical results, the WENO methodology itself may not be adequate to suppress solution overshoots. To enhance robustness of the method, we propose to couple a nonlinear filter. Such a filter is constructed via the information from the WENO procedure and hence will not increase the cost significantly. Moreover, we apply the WENO quadrature only for approximating operators with $p=1$ in \eqref{eq:H} and \eqref{eq:H2}, and use cheap high order linear quadrature for those with $p>1$. The numerical evidence indicates that by doing so we can reduce the cost and the scheme is still high order accurate and free of oscillations.

\subsection{WENO-based quadrature}

Below, the fifth order WENO-based quadrature for approximating $J^{L}_{i}=J^{L}[v,\alpha](x_{i})$ is provided as an example. The corresponding stencil used is shown in Figure \ref{Fig0}, and all coefficients are given in the Appendix. The process to obtain $J^{R}_{i}$ is mirror symmetric to that of $J^{L}_{i}$ with respect to point $x_{i}$.  
\begin{figure}
	\centering
	\includegraphics[width=0.5\textwidth]{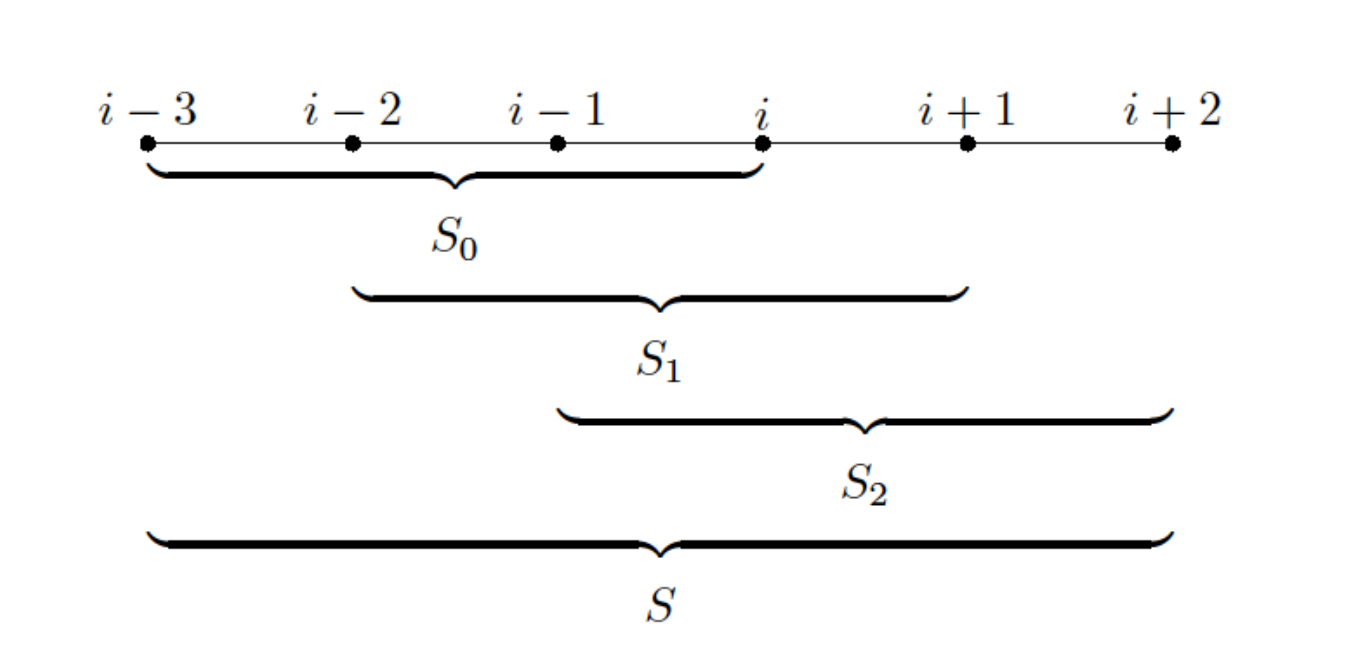}
	\caption{\em The structure of the stencils in WENO integration.   }
	\label{Fig0}
\end{figure}

\begin{enumerate}
	\item
	As with the standard WENO methodology, we first choose the three small stencils as
	$S_{r}(i)=\{x_{i-3+r},x_{i-2+r},x_{i-1+r},x_{i+r}\}$, $r=0, \,1,\, 2$. On each small stencil, there is a unique polynomial $p_{r}(x)$ of degree at most three which interpolates $v(x)$ at the nodes in $S_{r}(i)$. Then we are able to compute three candidates for $J^L_i$ denoted by $J^L_{i,r}$, $r=0,\,1,\,2$ 
	\begin{equation*}
	J^{L}_{i,r} =\alpha \int_{x_{i-1}}^{x_{i}}e^{-\alpha(x_{i}-y)}p_{r}(y)dx
	=\sum_{j=0}^{3}c^{(r)}_{-3+r+j}v_{i-3+r+j},
	\end{equation*}
	where the coefficients $c^{(r)}_{-3+r+j}$ depend on $\alpha$ and the cell size $\Delta x$, but not on $v$.
	
	\item
	On the entire big stencil $S(i)=\{x_{i-3},\ldots,x_{i+2}\}$, there is a unique polynomial $p(x)$ of degree at most five interpolating $v(x)$ at the nodes in $S(i)$. Then we have
	\begin{equation}
	J^{L}_{i,S} =\alpha \int_{x_{i-1}}^{x_{i}}e^{-\alpha(x_{i}-y)}p(y)dx
	=\sum_{j=0}^{5}c_{-3+j}v_{i-3+j}
	=\sum_{r=0}^{3}d_{r}J^{L}_{i,r}
	\end{equation}
	that approximates $J^{L}_{i}$ with the linear weights $d_{r}$.
	
	\item
	Replace the linear weights $d_{r}$ with the nonlinear weights $\omega_{r}$ that are defined as
	\begin{equation}
	\omega_{r}=\frac{\tilde{\omega}_{r}}{\sum_{s=0}^{2}\tilde{\omega}_{s}}, \ \ 
	\text{with} \ \ 
	\tilde{\omega}_{r}=\frac{d_{r}}{(\epsilon+SI_{r})^2}, \ \ r=0,\,1,\,2.
	\end{equation}
	Here, $\epsilon>0$ is a small number to avoid a zero denominator, and we take $\epsilon=10^{-6}$ in our numerical tests. The smoothness indicators $SI_{r}$ that measure the relative smoothness of the function $v(x)$ in the stencil $S_{r}(i)$ are defined as
	\begin{align*}
	SI_{0}=& \frac{781}{720}(-v_{i-3}+3v_{i-2}-3v_{i-1}+v_{i})^2 
	+ \frac{13}{48}(v_{i-3}-5v_{i-2}+7v_{i-1}-3v_{i})^2 
	+ (v_{i-1}-v_{i})^2,\\
	SI_{1}=& \frac{781}{720}(-v_{i-2}+3v_{i-1}-3v_{i}+v_{i+1})^2 
	+ \frac{13}{48}(v_{i-2}-v_{i-1}-v_{i}+v_{i+1})^2 
	+ (v_{i-1}-v_{i})^2,\\
	SI_{2}=& \frac{781}{720}(-v_{i-1}+3v_{i}-3v_{i+1}+v_{i+2})^2 
	+ \frac{13}{48}(-3v_{i-1}+7v_{i}-5v_{i+1}-v_{i+2})^2
	+ (v_{i-1}-v_{i})^2.
	\end{align*}
	
	\item
	Lastly, $J^{L}_{i}$ is approximated by $J^{L}_{i,W}$, where
	$ \displaystyle
	J^{L}_{i,W}=\sum_{r=0}^{2}\omega_{r}J^{L}_{i,r}$.
	
\end{enumerate}

\subsection{Nonlinear filter}

As mentioned above, the $k^{th}$ order accuracy in time is attained via the $k^{th}$ partial sum. However, it is observed from the numerical results that,  when $k\geq2$, 
spurious oscillations will appear for some non-smooth problems, even though the WENO quadrature is applied to compute the partial sum $\mathcal{H}$. 
%
Therefore, to further enhance robustness and to avoid spurious oscillations, we introduce a nonlinear ``filter'' denoted by $\sigma$ when approximating $\partial_{x}$. Note that such a filter is only needed for the convection part. Below, we only consider periodic boundary conditions to illustrate the idea and the proposed methodology can be extended straightforwardly to handle the special homogeneous boundary condition. 

The key idea of the proposed nonlinear filter is that, unlike \eqref{eq:partial}, we use the following modified formulation to approximate $\partial_x$:
\label{eq:highorder}
\begin{align}
 \partial_{x}\approx \frac{\beta}{c\Delta t}\mathcal{D}_{L} + \frac{\beta}{c\Delta t} \sum_{p=2}^{k}\sigma_{L,i}^{p-1}\mathcal{D}_{L}^p, \quad\text{and}\quad
 \partial_{x}\approx -\frac{\beta}{c\Delta t}\mathcal{D}_{R} - \frac{\beta}{c\Delta t} \sum_{p=2}^{k}\sigma_{R,i}^{p-1}\mathcal{D}_{R}^p,
\end{align}
where the filters $\sigma_{L,i}$ and $\sigma_{R,i}$ are incorporated. In this work,  
  the filters are designed to fulfill several requirements: (a) 
 $\sigma_{L,i}$ and $\sigma_{R,i}$ are $1+\mathcal{O}(\Delta x^k)$ when the solution is continuous, thus maintaining the original high order accuracy;
  (b) they are close to 0 when the grid point $x_i$ is in the vicinity of a discontinuity, thus decreasing the influence of the high order terms which may lead to oscillations;
  (c) the associated computational overhead is relatively low. 
%
%
  To achieve this goal, the design of the filter is
  based on the smoothness indicators from the WENO methodology and the underlying idea is similar to that proposed in \cite{borges2008improved}.

Below, we provide the details of construction of the filter.
Assume that we have obtained the approximation to the integral $J^{L}_{i}$ based on the WENO-based quadrature using the six-point stencil $S(i)=\{x_{i-3},\ldots,x_{i+2}\}$ and the associated  three small stencils $S_{0}(i)$, $S_{1}(i)$ and $S_{2}(i)$ given in Section 4.1.
Following the idea in \cite{borges2008improved}, we introduce a new parameter $\tau_i$, which is simply defined as the absolute difference between the smoothness indicators $SI_{0}$ and $SI_{2}$, namely,
$$\tau_i=|SI_{0}-SI_{2}|.$$
Note that $\tau_i$ can be obtained with little cost since $SI_{0}$ and $SI_{2}$ are already available.
If $v(x)$ is smooth on the entire stencil $S(i)$, applying the Taylor expansion to $SI_{0}$ and $SI_{2}$ gives
\begin{align*}
SI_{0}
=& (v'_{i-1/2})^2 \Delta x^2 + \frac{1}{12}\big( 13(v''_{i-1/2})^2+v'_{i+1/2}v^{(3)}_{i-1/2} \big) \Delta x^4 + \frac{1}{2880}\big( 3129(v^{(3)}_{i-1/2})^2\\
& -1820v''_{i-1/2}v^{(4)}_{i-1/2} + 3v'_{i-1/2}v^{(5)}_{i-1/2} \big)\Delta x^6 +\mathcal{O}(\Delta x^7), \\
SI_{2}
=& (v'_{i-1/2})^2 \Delta x^2 + \frac{1}{12}\big( 13(v''_{i-1/2})^2+v'_{i+1/2}v^{(3)}_{i-1/2} \big) \Delta x^4 + \frac{1}{2880}\big( 3129(v^{(3)}_{i-1/2})^2\\
& -1820v''_{i-1/2}v^{(4)}_{i-1/2} + 3v'_{i-1/2}v^{(5)}_{i-1/2} \big)\Delta x^6 +\mathcal{O}(\Delta x^7) .
\end{align*}
Thus, we deduce that
$\tau_i=\mathcal{O}(\Delta x^7).$
We further define
$$\xi_{i}= \frac{1+\tau_i^2/(SI_{max}+\epsilon)^2}{1+\tau_i^2/(SI_{min}+\epsilon)^2},$$
where
$$SI_{max}=\max(SI_{0},SI_{2}), 
\ \ \ \text{and} \ \ \ 
SI_{min}=\min(SI_{0},SI_{2}).$$
Note that
$SI_{max}$ and $SI_{min}$ are both $\mathcal{O}(\Delta x^2)$ in a monotone region, and $\mathcal{O}(\Delta x^4)$ near a critical point. A simple Taylor expansion applying to $\xi_{i}$ yields
$$\xi_{i} =1+\mathcal{O}(\Delta x^6).$$
Here, we take $\epsilon=10^{-6}$ to avoid a zero denominator. On the other hand, if the solution $v(x)$ contains a discontinuity within the interval $[x_{i-3},x_{i-1}]$ or $[x_{i},x_{i+2}]$, we can show that $\xi_{i}=\mathcal{O}(\Delta x^4)$ via a Taylor expansion. Meanwhile, if $v(x)$ is discontinuous within $[x_{i-1},x_{i}]$, then $\xi_i$ may be $\mathcal{O}(1)$, while we have $\xi_{i+1}=\mathcal{O}(\Delta x^4)$ which is defined at the neighboring grid point $x_{i+1}$. The nonlinear filter is defined as 
\begin{align*}
& \sigma_{L,i}=\min(\xi_{i},\xi_{i+1}).
\end{align*}	
$\sigma_{R,i}$ is mirror symmetric to $\sigma_{L,i}$ with respect to $x_i$, and it reads
\begin{align*}
\sigma_{R,i}=\min(\xi_{i-1},\xi_{i}),
\end{align*}
where $\xi_{i}$ is obtained through the smoothness indicators that are used for computing $J^{R}_{i,W}$. 
\section{Time discretization and stability}
In this section, we introduce the time discretization methods for evolving \eqref{eq:ad} based on the partial sum approximation \eqref{eq:H}, and then analyze the linear stability property. Denote $u^{n}$ as the semi-discrete solution at time $t^{n}$. In this work, we propose to use the classic explicit SSP RK methods \cite{gottlieb2001strong} to advance $u^{n}$ to $u^{n+1}$. For example, the first order scheme is the forward Euler scheme
\begin{align}
\label{eq:rk1}
u^{n+1}=u^{n}+\Delta t \mathcal{H}[u^{n}].
\end{align} 
The second order SSP RK scheme is given as
\begin{align}
\label{eq:rk2}
& u^{(1)}=u^{n}+\Delta t \mathcal{H}[u^{n}],\nonumber\\
& u^{n+1}=\frac{1}{2}u^{n}+\frac{1}{2}\left( u^{(1)} +\Delta t \mathcal{H}[u^{(1)}] \right).
\end{align}
And the third order SSP RK scheme is given as
\begin{align}
\label{eq:rk3}
& u^{(1)}=u^{n}+\Delta t \mathcal{H}[u^{n}],\nonumber\\
& u^{(2)}=\frac{3}{4}u^{n}+\frac{1}{4} \left( u^{(1)}+\Delta t \mathcal{H}[u^{(1)}] \right), \nonumber\\
& u^{n+1}=\frac{1}{3}u^{n}+\frac{2}{3} \left( u^{(2)}+\Delta t \mathcal{H}[u^{(2)}] \right).
\end{align}

Note that, to achieve $k^{th}$ order accuracy in time, we should employ the $k^{th}$ order SSP RK method as well as the  $k^{th}$ partial sum $\mathcal{H}[u]$. A remarkable advantage of the proposed scheme is that, even though the explicit SSP RK method is used for time integration, unlike the standard MOL approach, the scheme can be A-stable and hence allowing for large time step evolution if $\beta$ in \eqref{eq:H} is appropriately chosen. In particular, we establish linear stability of the scheme in the following theorem.

\begin{thm}\label{thm4}
	(a) For the linear advection equation $u_{t}+cu_{x}=0$ with periodic boundary conditions, there exists constant $\beta_{1,k,max}>0$ for $k=1,\,2$, such that the scheme is A-stable provided  $0<\beta\leq\beta_{1,k,\max}$;
	
	(b) For the linear diffusion equation $u_{t}=bu_{xx}$ with $b>0$ and periodic boundary conditions, there exists constant $\beta_{2,k,max}>0$ for $k=1,\,2,\,3$, such that the scheme is A-stable provided  $0<\beta\leq\beta_{2,k,\max}$. 
	
	The constants $\beta_{1,k,max}$ and $\beta_{2,k,\max}$ are summarized in Table \ref{tab0}.
\end{thm}

\begin{proof} Here, we only show the proof for $k=1$ for brevity. Given the ansatz $u^{n}=\hat{u}^{n}e^{i\kappa x}$, we can obtain the amplification factor $\lambda$ via a Von Neumann analysis. The scheme is unconditionally stable if  $|\lambda|\leq 1$ for any $\kappa$ and time step $\Delta t$. 
	\begin{enumerate}
		\item We present the proof for the case $c>0$. For $c<0$, the proof can be established in the same way. Upon the definitions of $\mathcal{D}_{L}$ and $\mathcal{L}_{L}$, by taking the Fourier transform in space, we obtain that $\widehat{\mathcal{L}}_{L}=1+(i\kappa)/\alpha_{L}$, and then
		$$\widehat{\mathcal{D}}_{L} =1-1/\widehat{\mathcal{L}}_{L} =\frac{i\kappa/\alpha_{L}}{1+i\kappa/\alpha_{L}}.$$
	    For the forward Euler scheme $u^{n+1}=u^{n}-\Delta t\, \alpha_{L}\, \mathcal{D}_{L}[cu^{n},\alpha_{L}]$ with the parameter $\alpha_{L}=\beta/(c\Delta t)$, we could compute the amplification factor $\lambda$
		$$\lambda=1-\beta\frac{i\kappa c\Delta t/\beta}{1+i\kappa c\Delta t/\beta}.$$
		Then, we have $|\lambda|\leq1$ when $\beta\leq2$, which implies the scheme is A-stable. Hence, for the first order scheme $k=1$, we can choose $\beta_{1,1,max}=2$.

		\item Similarly, for the forward Euler scheme $u^{n+1}=u^{n}-\Delta t\, \alpha_{0}^2 \, \mathcal{D}_{0}[bu^{n},\alpha_{0}]$ and $\alpha_{0}=\sqrt{\beta/b \Delta t}$, the amplification factor $\lambda$ is 
		$$\lambda = 1-\beta\widehat{\mathcal{D}}_{0} \ \ 
		\text{with} \ \ 
		\widehat{\mathcal{D}}_{0} =1-1/\widehat{\mathcal{L}}_{0} =\frac{(\kappa/\alpha_{0})^2}{1+(\kappa/\alpha_{0})^2}\in[0,1].$$  
		Then, we still have $\beta\leq2$ to ensure $|\lambda|\leq1$. Thus, we let $\beta_{2,1,\max}=2$.
	\end{enumerate}	
\end{proof}

\begin{rem}
	\label{rem1}
	Note that, for the linear advection equation $u_{t}+cu_{x}=0$ with periodic boundary conditions, the third order scheme \eqref{eq:H} can only be A($\alpha$)-stable. But, fortunately, we find that if the scheme is modified as
	\begin{align}
	\label{eq:H2}
	\mathcal{H}[u](x)=& -\frac{\beta}{c\Delta t} \sum_{p=1}^{3}\mathcal{D}_{L}^{p}[f^{+}(u),\frac{\beta}{c\Delta t}](x) 
	+ \frac{\beta}{c\Delta t} \sum_{p=1}^{3}\mathcal{D}_{R}^{p}[f^{-}(u),\frac{\beta}{c\Delta t}](x) \nonumber\\
	& -\frac{\beta}{b\Delta t}\sum_{p=1}^{3}\mathcal{D}_{0}^{p} [g(u),\sqrt{\frac{\beta}{b\Delta t}}](x)\nonumber\\
	& + \frac{\beta}{c\Delta t} \mathcal{D}_{0}[ \, \mathcal{D}_{L}^{2}[f^{+}(u),\frac{\beta}{c\Delta t}]-\mathcal{D}_{L}^{2}[f^{-}(u),\frac{\beta}{c\Delta t}]\, ,\frac{\beta}{c\Delta t}](x),	
	\end{align}
	still with periodic boundary treatment for the last term,
	then the scheme coupled with the third order SSP RK integrator is also A-stable provided $0<\beta\leq\beta_{1,3,\max}$. 
   In light of Lemma \ref{lem1}, the extra term in \eqref{eq:H2} is in fact an approximation to $f_{xxxx}$. It will enhance the stability of the scheme and make 	$\mathcal{H}[u](x)$  fourth order accurate for the case $b\ll c$, i.e., convection dominates. For the homogeneous boundary condition case, we can similarly add the extra term with the treatment \eqref{eq:gxxhomo} and make the scheme A-stable. The parameter $\beta_{1,3,\max}$ is given in Table \ref{tab0}.
\end{rem}

	Below, we provide a Fourier analysis for the fully discrete scheme with the sixth order linear quadrature rule
	\begin{align}
	\label{eq:linear}
	J^{L}_{j}=\sum_{r=-3}^{2}c_{r}v_{j+r}, \quad \text{and} \quad J^{R}_{j}=\sum_{r=-3}^{2}c_{r}v_{j-r}.
	\end{align}
	Without loss of generality, we take $c=1$ and $b=1$ for simplicity.
	Under the assumption that $u^{n}_{j}=\hat{u}^{n}e^{i\kappa x_{j}}$, we can obtain
	\begin{align}
	\hat{\mD}_{L}=1- \frac{\sum_{r=-3}^{2} c_{r}e^{i r\kappa\Delta x}}{1-e^{-\alpha\Delta x-i \kappa\Delta x}}
	\end{align}
	and
	\begin{align}
	\hat{\mD}_{0}=1-\frac{1}{2}\left( \frac{\sum_{r=-3}^{2} c_{r}e^{i r\kappa\Delta x}}{1-e^{-\alpha\Delta x-i \kappa\Delta x}} + \frac{\sum_{r=-3}^{2} c_{r}e^{-i r\kappa\Delta x}}{1-e^{-\alpha\Delta x+i \kappa\Delta x}}\right)
	\end{align}
	Moreover, it is straightforward to check that the amplification factor $\lambda$ for the linear advection equation $u_{t}+u_{x}=0$ depends on $\beta$, $\kappa\Delta x$ and $\Delta t/\Delta x$, while for the linear diffusion equation $u_{t}=u_{xx}$, $\lambda$ depends on $\beta$, $\kappa\Delta x$ and $\Delta t/\Delta x^2$. 
	Even though it is very tedious and difficult to derive analytically the condition of $\lambda\le1$, 
	as a common practice, we can still numerically verify that, if $0<\beta<\beta_{\cdot,k,\max}$, for $k=1,\, 2,\, 3$, then $|\lambda|\le1$ for any $\kappa\Delta x\in[0,2\pi ]$, $\Delta t$ and $\Delta x$. In Figure \ref{Fig1add}, we plot the contours of $|\lambda|$ with $\beta=\beta_{\cdot,k,\max}$ to justify this property. In other words, the scheme is unconditionally stable if $\beta$ is appropriately chosen according to Table \ref{tab0}.

\begin{figure}
	\centering
	\subfigure[$k=1$. $\beta=2$.]{
		\includegraphics[width=0.3\textwidth]{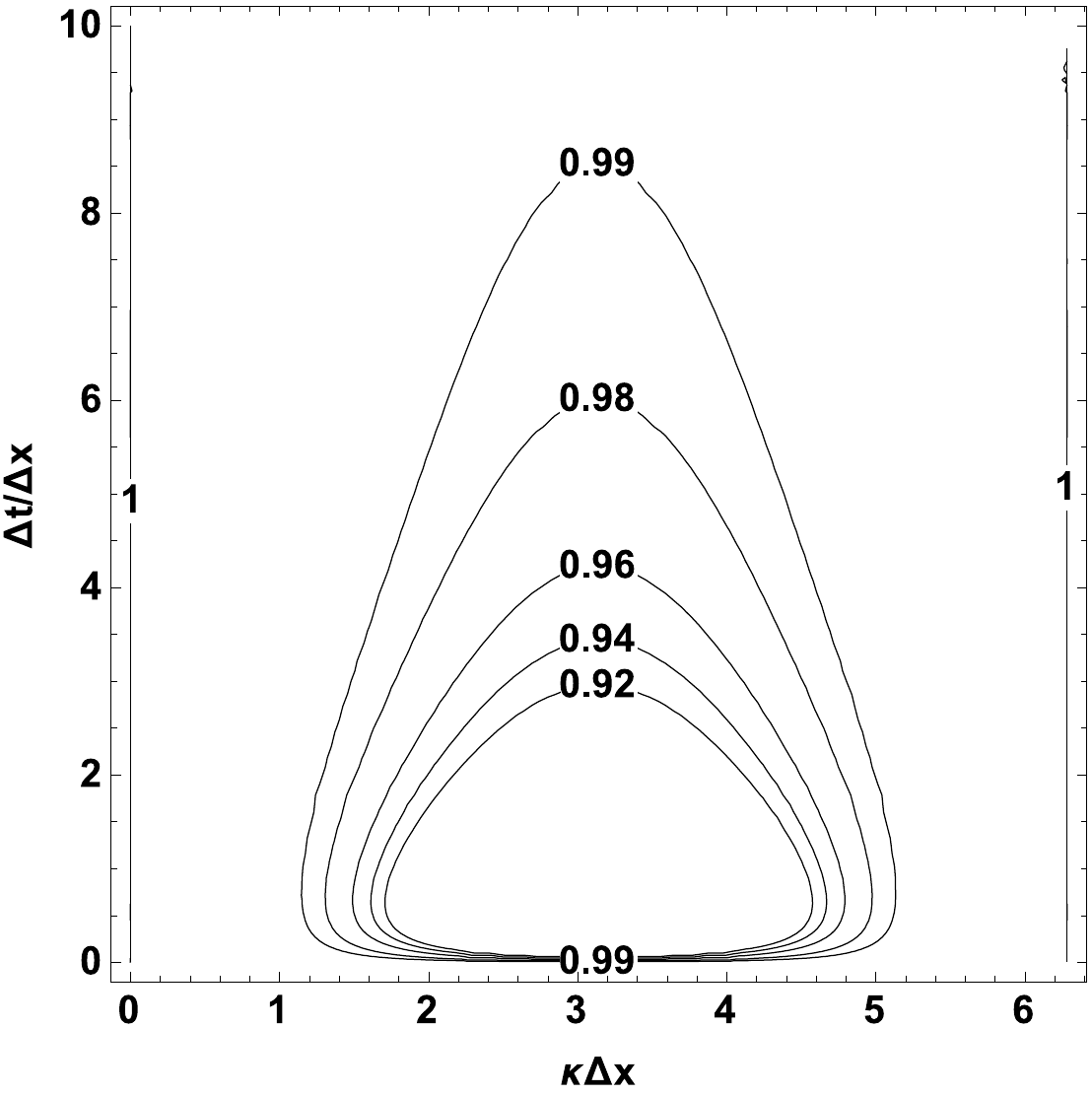}\label{Fig1add.1}}
	\subfigure[$k=2$. $\beta=1$.]{
		\includegraphics[width=0.3\textwidth]{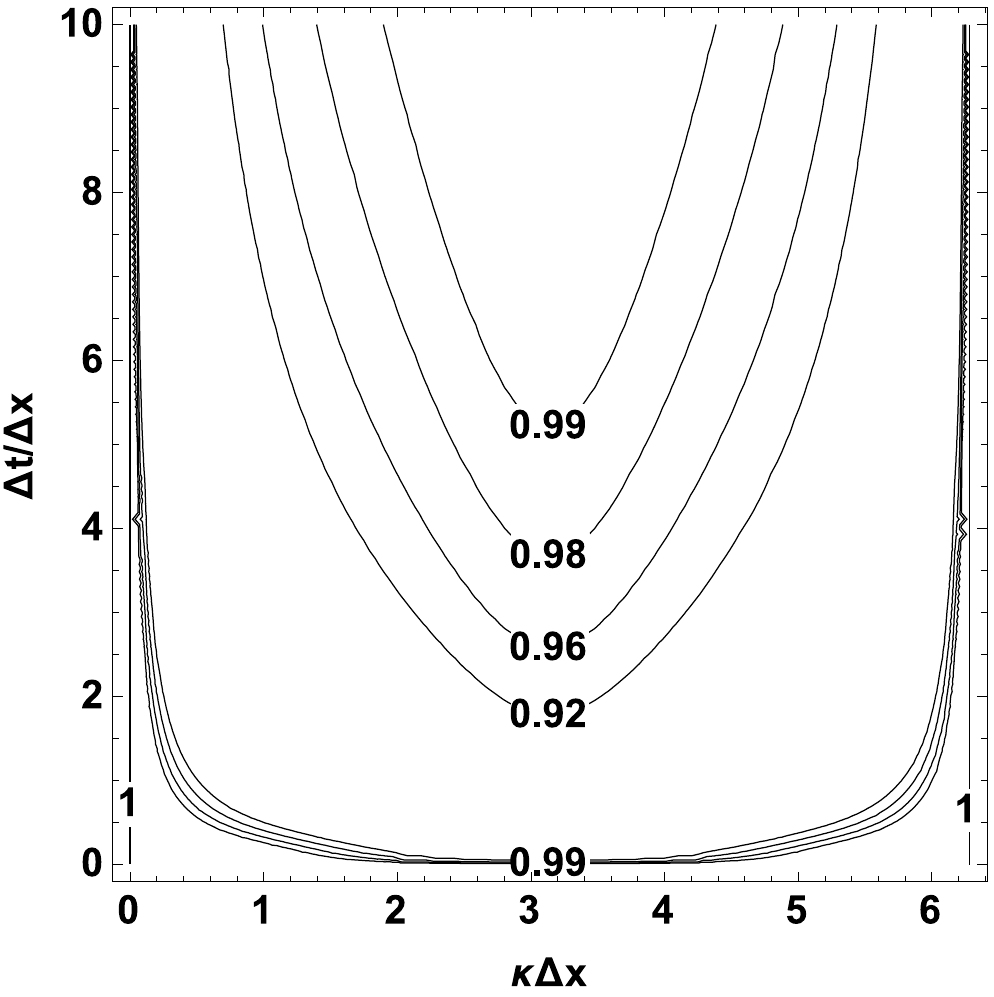}\label{Fig1add.2}}
	\subfigure[$k=3$. $\beta=1.243$.]{
		\includegraphics[width=0.3\textwidth]{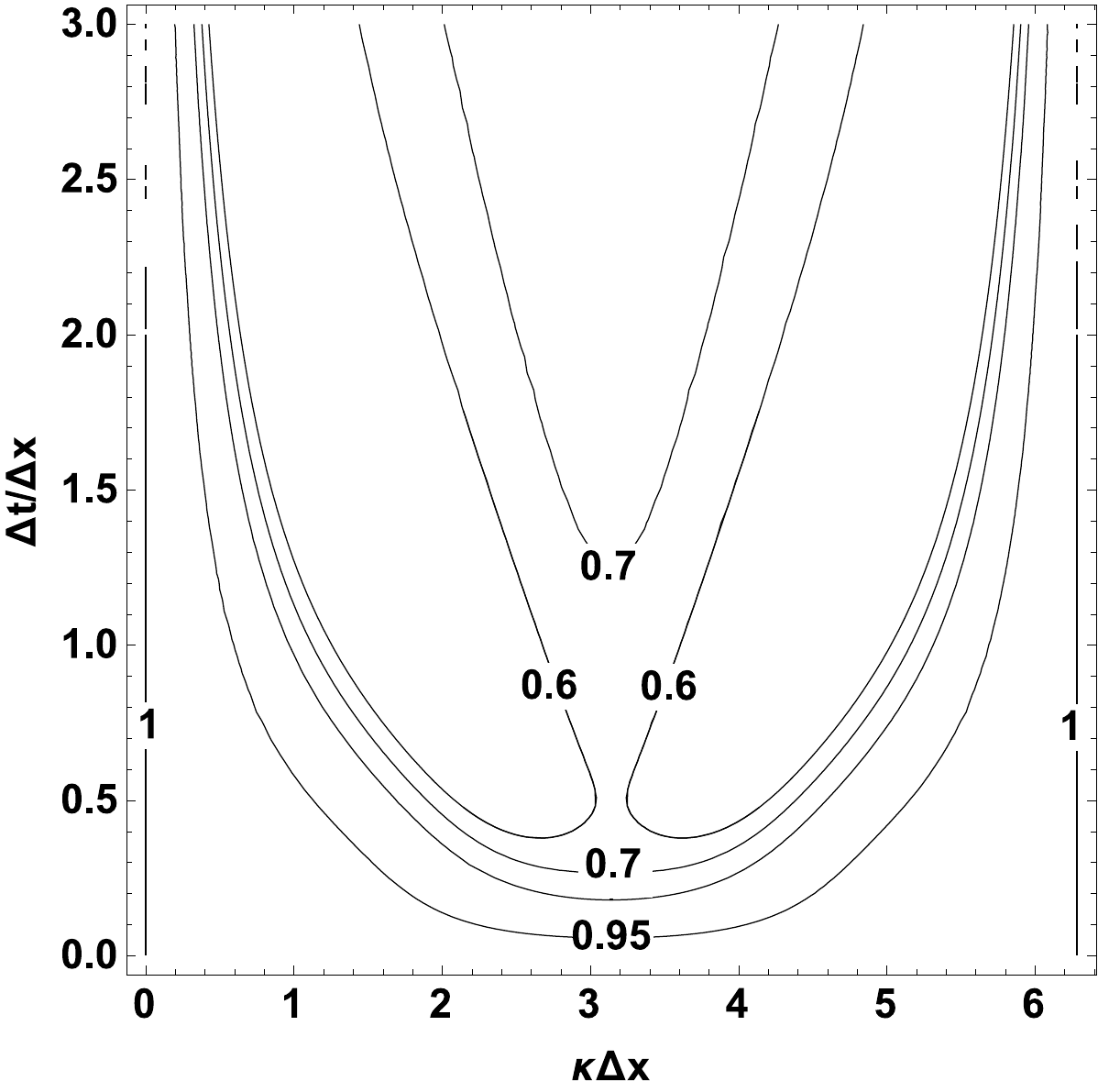}\label{Fig1add.3}}
	\subfigure[$k=1$. $\beta=2$.]{
		\includegraphics[width=0.3\textwidth]{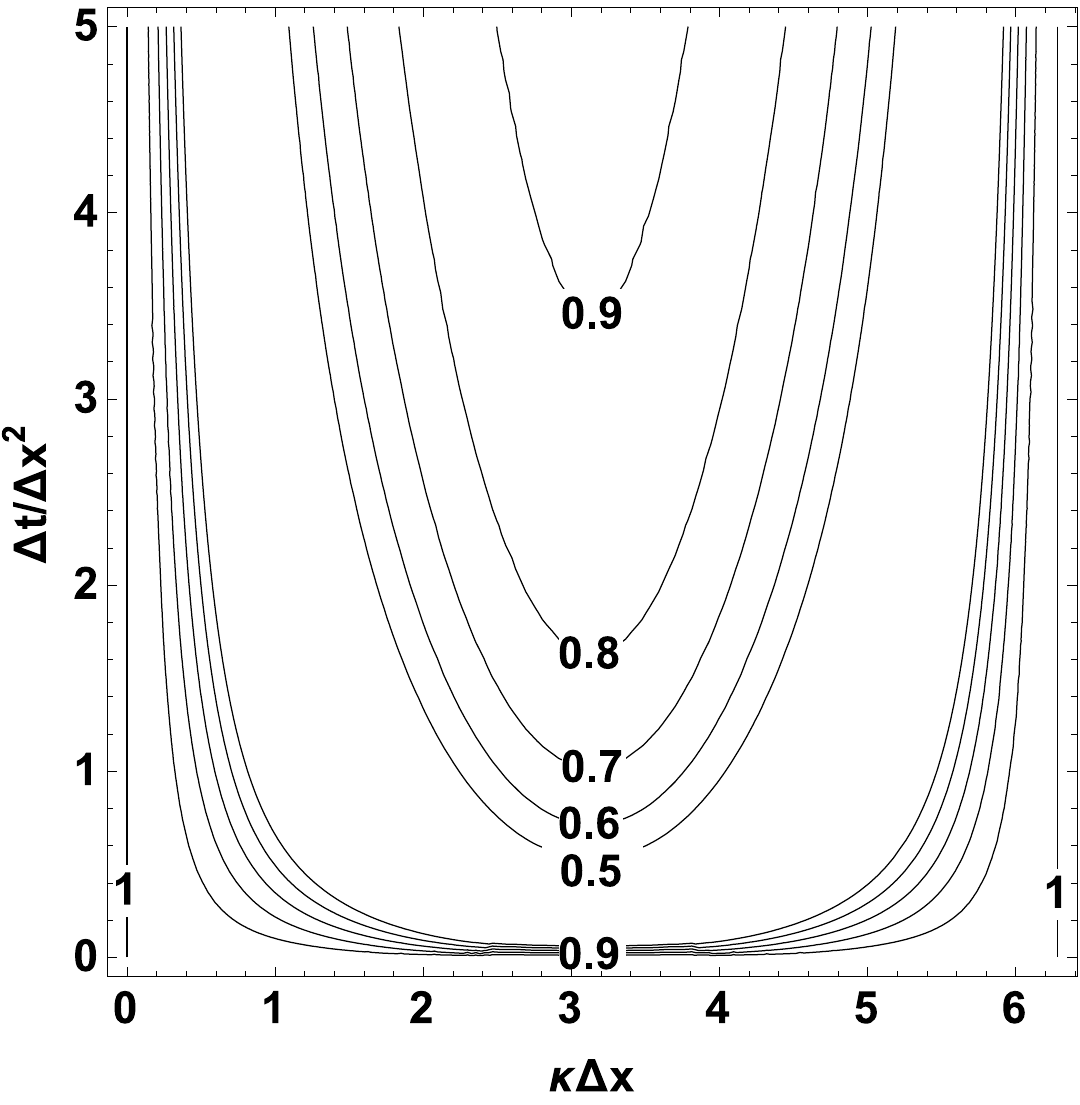}\label{Fig1add.4}}
	\subfigure[$k=2$. $\beta=1$.]{
		\includegraphics[width=0.3\textwidth]{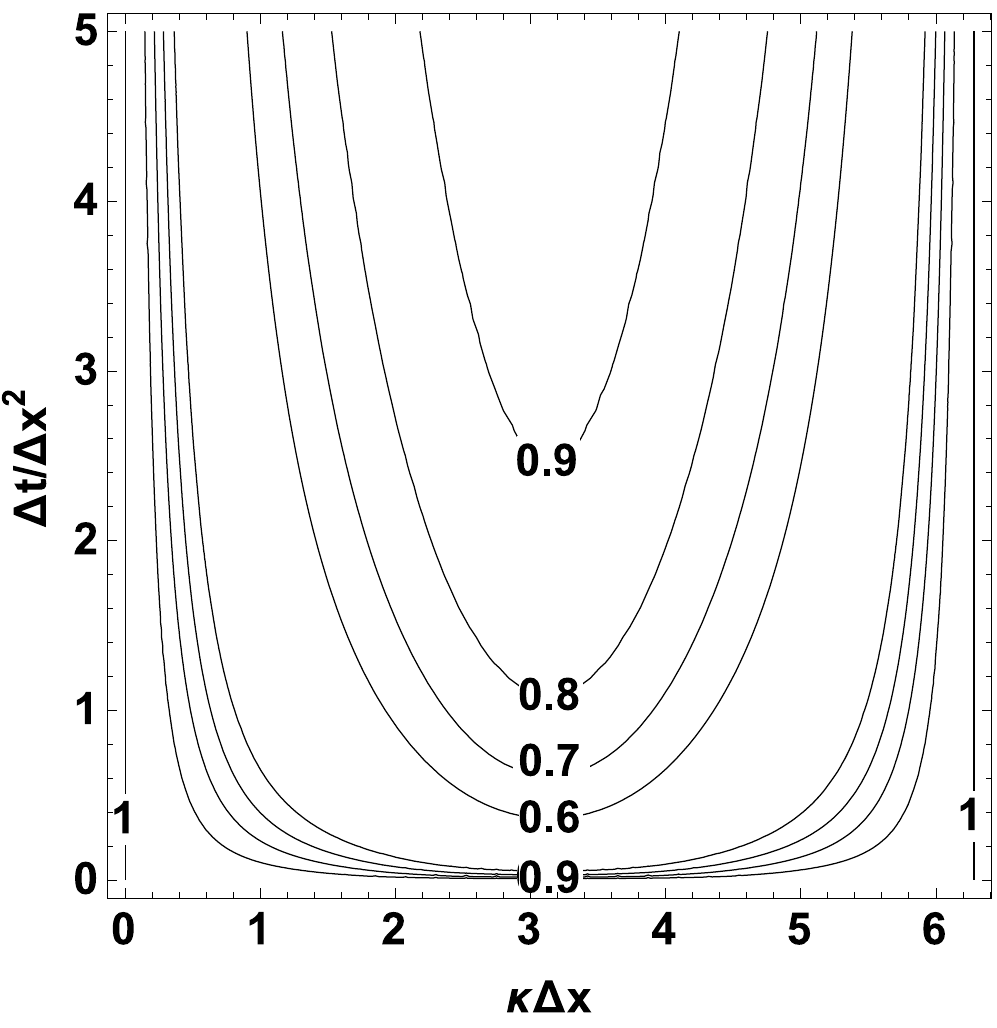}\label{Fig1add.5}}
	\subfigure[$k=3$. $\beta=0.8375$.]{
		\includegraphics[width=0.3\textwidth]{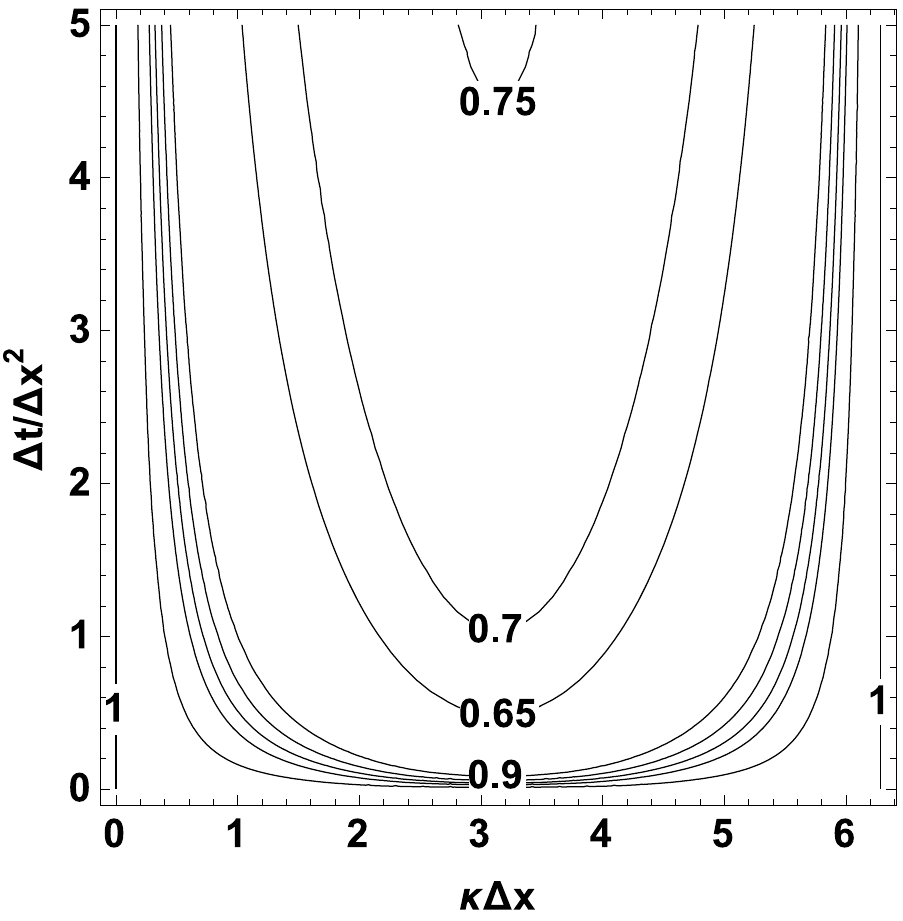}\label{Fig1add.6}}
	\caption{\em Contours of $|\lambda|$. Upper: linear advection equation $u_{t}+u_{x}=0$; below: linear diffusion equation $u_{t}=u_{xx}$. }
	\label{Fig1add}
\end{figure}

Combining the both cases in Theorem \ref{thm4}, Remark \ref{rem1} and above analysis, we can easily establish a similar unconditional stability property of the scheme for solving linear convection-diffusion problems.

\begin{thm}\label{thm3}
Consider the linear convection-diffusion problem 
\begin{align}
\label{eq:linearad}
u_{t}+c\,u_{x}=b\,u_{xx},
\end{align}
with periodic boundary conditions.
$c$ and $b$ are both constant and $b\geq0$. Suppose the scheme employs the $k^{th}$ order SSP RK method, the $k^{th}$ partial sum in \eqref{eq:H} or \eqref{eq:H2} for  $k=1,\,2,\,3$, and the linear quadrature rule \eqref{eq:linear}. Then, the scheme is unconditionally stable if $0<\beta\leq\beta_{k,\max}$, where $\beta_{k,max}:=\frac{1}{2} \min(\beta_{1,k,max},\beta_{2,k,\max})$, and the constants
$\beta_{1,k,max}$ and $\beta_{2,k,\max}$ are summarized in Table \ref{tab0}. 
\end{thm}

\begin{table}[htb]
\caption{\label{tab0}\em $\beta_{\max}$ in Theorem \ref{thm4}, Remark \ref{rem1} and Theorem \ref{thm3} for  $k=1,\,2,\,3$.}
\centering
\begin{tabular}{|c|c|c|c|}
  \hline
  k  &  $\beta_{1,k,max}$ &  $\beta_{2,k,max}$  & $\beta_{k,max}$  \\\hline
  1  &  2  &  2  &  1     \\
  2  &  1  &  1  &  0.5     \\
  3  & 1.243  &  0.8375  &  0.4167  \\\hline
  \end{tabular}
\end{table}


\section{Two-dimensional implementation}
\label{section:2D}

Consider the following two-dimensional problem
\begin{equation}
\label{eq:2d}
u_{t}+f_{1}(u)_{x}+f_2(u)_{y}=g_{1}(u)_{xx}+g_2(u)_{yy}.
\end{equation}
The proposed one-dimensional formulation can be directly extended to solving \eqref{eq:2d} based on a dimension-by-dimension approach, namely, approximating $\partial_{x}$ and $\partial_{xx}$ for fixed $y_j$ and approximating $\partial_{y}$ and $\partial_{yy}$ for for fixed $x_i$. More specifically, for periodic boundary conditions
\begin{align*}
& g_{1}(u)_{xx}|_{(x_{i},y_{j})}\approx -\alpha^2_{0,x}\sum_{p=1}^{k}\mathcal{D}_{0}^{p} [g_1(\cdot,y_{j}),\alpha_{0,x}](x_{i}),\quad g_{2}(u)_{yy}|_{(x_{i},y_{j})}\approx -\alpha^2_{0,y}\sum_{p=1}^{k}\mathcal{D}_{0}^{p} [g_2(x_{i},\cdot),\alpha_{0,y}](y_{j}),
\end{align*}
where $\alpha_{0,x}=\sqrt{b_x/(\beta\Delta t)}$ and $\alpha_{0,y}=\sqrt{b_y/(\beta\Delta t)}$ with
$b_{x}=\max_{u}|g'_1(u)|$, $b_{y}=\max_{u}|g'_2(u)|$.
To approximate $f_{1}(u)_{x}$ and $f_{2}(u)_y$, the flux splitting strategy is still needed:
\begin{align*}
& f_{1}^{\pm}(u)=\frac{1}{2}(f_{1}(u)\pm c_{x}u), \quad
 f_{2}^{\pm}(u)=\frac{1}{2}(f_{2}(u)\pm c_{y}u), 
\end{align*}
where $c_{x}=\max_{u}|f'_{1}(u)|$ and  $c_{y}=\max_{u}|f'_{2}(u)|$.
Then, the dimension-by-dimension approach can be similarly applied. Again, for periodic boundary conditions, 
\begin{align*}
& f_{1}(u)_{x}|_{(x_{i},y_{j})}\approx -\alpha_{L,x}\sum_{p=1}^{k}\mathcal{D}_{L}^{p} [f_1^{+}(\cdot,y_{j}),\alpha_{L,x}](x_{i}) + \alpha_{R,x}\sum_{p=1}^{k}\mathcal{D}_{R}^{p} [f_1^{-}(\cdot,y_{j}),\alpha_{R,x}](x_{i}),\\
& f_{2}(u)_{y}|_{(x_{i},y_{j})}\approx -\alpha_{L,y}\sum_{p=1}^{k}\mathcal{D}_{L}^{p} [f^{+}_2(x_{i},\cdot),\alpha_{L,y}](y_{j}) + \alpha_{R,y}\sum_{p=1}^{k}\mathcal{D}_{R}^{p} [f^{-}_2(x_{i},\cdot),\alpha_{R,y}](y_{j}),
\end{align*}
or with a modified term for $k=3$. In addition, in the $x$-direction, we choose $\alpha_{L,x}=\alpha_{R,x}=\beta/(c_{x}\Delta t)$, and in the $y$-direction, we choose $\alpha_{L,y}=\alpha_{R,y}=\beta/(c_y\Delta t )$. 

In the two-dimensional case,  $\beta_{max}$ needs to be chosen as half of that for one-dimensional problems to attain the unconditional stability for the scheme. 
\section{Numerical results}
In this section, we present the numerical results to demonstrate  efficiency and efficacy of the proposed scheme. For one-dimensional problems, we choose the time step as
\begin{align*}
\Delta t=\text{CFL}\frac{\Delta x}{b+c},
\end{align*}
while for two-dimensional problems, the time step is set as
\begin{align*}
\Delta t=\frac{\text{CFL}}{\max((b_{x}+c_{x})/\Delta x,(b_{y}+c_{y})/\Delta y)}.
\end{align*}
We remark that, if the problem does not have an analytical solution, we will use the numerical solutions by the following first order numerical scheme
$$u^{n}_{i} = u^{n}_{i} - \frac{\Delta t}{\Delta x}\left( f^{+}_{i}-f^{+}_{i-1}\right) - \frac{\Delta t}{\Delta x}\left( f^{-}_{i+1}-f^{-}_{i}\right) + \frac{\Delta t}{\Delta x^2}\left( g_{i+1} - 2g_{i} + g_{i-1}\right)$$
 with $N=3000$ grid points and $\Delta t = 0.1\Delta x^2 /(c\Delta x+2b)$ as a reference solution.

\vspace{0.5cm}

\textbf{Example 1.}
 We test the accuracy of the scheme for the one-dimensional linear advection-diffusion problem
\begin{align}
\left\{\begin{array}{ll}
u_{t}+c\,u_{x}=b\,u_{xx}, & -\pi\leq x\leq\pi,\\
u(x,0)=\sin(x),
\end{array}
\right.
\end{align}
with the $2\pi$-periodic boundary condition.
Here, $c$ and $b\geq0$ are given constants. This problem has the exact solution
$u^{e}(x,t)=e^{-bt}\sin(x-ct).$

In Table \ref{tab1}-\ref{tab2}, we summarize the convergence study for the case of $c=1, b=0.01$ and $c=1, b=1$ at final time $T=2$, and the $L_{\infty}$ errors and the associated orders of accuracy are provided.
The analysis presented in Section 3 is verified that the use of the $k^{th}$ partial sum yields $k^{th}$ order accuracy. Moreover, the scheme allows for large CFL numbers due to its unconditionally stability.

\begin{table}[htb]
	\caption{\label{tab1}\em Example 1: $L_\infty$ errors and orders of accuracy at $T=2$. $c=1$ and $b=0.01$.  }
	\centering
	\begin{small}
		\begin{tabular}{|c|c|cc|cc|cc|}
			\hline
			\multirow{2}{*}{CFL} &  \multirow{2}{*}{$N_x$} & \multicolumn{2}{c|}{$k=1$. $\beta=1$.} & \multicolumn{2}{c|}{$k=2$. $\beta=0.5$.} & \multicolumn{2}{c|}{$k=3$. $\beta=0.4$.}\\
			\cline{3-8}
			& &  error &   order  &  error &  order  &  error  & order  \\\hline
			\multirow{5}{*}{0.5}
			&  40   &  7.260E-02  &     --      &  4.729E-02  &     --     &  2.559E-03  &     --      \\
			&  80   &  3.715E-02  &  0.967  &  1.218E-02  &  1.957  &  1.712E-04  &  3.902  \\
			& 160  &  1.885E-02  &  0.979  &  3.077E-03  &  1.985  &  1.091E-05  &  3.972  \\
			& 320  &  9.473E-03  &  0.992  &  7.703E-04  &  1.998  &  6.865E-07  &  3.990  \\
			& 640  &  4.750E-03  &  0.996  &  1.928E-04  &  1.999  &  4.357E-08  &  3.978  \\\hline	 				
			\multirow{5}{*}{1}
			&  40  &  1.388E-01  &     --      &  1.697E-01  &     --      &  3.263E-02   &     --      \\
			&  80  &  7.260E-02  &  0.935  &  4.729E-02  &  1.843  &  2.559E-03  &  3.672  \\
			& 160  &  3.717E-02  &  0.966  &  1.218E-02  &  1.956  &  1.712E-04  &  3.902  \\
			& 320  &  1.885E-02  &  0.980  &  3.077E-03  &  1.986  &  1.091E-05  &  3.973  \\
			& 640  &  9.473E-03  &  0.992  &  7.703E-04  &  1.998  &  6.864E-07  &  3.990 \\\hline 	   		           
			\multirow{5}{*}{2}
			&  40  &  2.474E-01  &     --      &  4.375E-01  &     --      &  2.313E-01  &     --      \\
			&  80  &  1.388E-01  &  0.834  &  1.697E-01  &  1.366  &  3.271E-02  &  2.822  \\
			& 160  &  7.260E-02  &  0.935  &  4.733E-02  &  1.842  &  2.561E-03  &  3.675  \\
			& 320  &  3.717E-02  &  0.966  &  1.218E-02  &  1.958  &  1.713E-04  &  3.902  \\
			& 640  &  1.885E-02  &  0.980  &  3.077E-03  &  1.986  &  1.091E-05  &  3.973 \\\hline	
		\end{tabular}
	\end{small}
\end{table}

\begin{table}[htb]
	\caption{\label{tab2}\em Example 1: $L_\infty$ errors and orders of accuracy at $T=2$. $c=1$ and $b=1$.  }
	\centering
	\begin{small}
		\begin{tabular}{|c|c|cc|cc|cc|}
			\hline
			\multirow{2}{*}{CFL} &  \multirow{2}{*}{$N_x$} & \multicolumn{2}{c|}{$k=1$. $\beta=1$.} & \multicolumn{2}{c|}{$k=2$. $\beta=0.5$.} & \multicolumn{2}{c|}{$k=3$. $\beta=0.4$.}\\
			\cline{3-8}
			& &  error &   order  &  error &  order  &  error  & order  \\\hline
			\multirow{5}{*}{0.5}
			&  40   &  1.047E-02  &     --      &  1.821E-03  &     --     &  1.912E-04  &     --      \\
			&  80   &  5.272E-03  &  0.990  &  4.953E-04  &  1.879  &  2.787E-05  &  2.779  \\
			& 160  &  2.646E-03  &  0.995  &  1.293E-04  &  1.937  &  3.751E-06  &  2.893  \\
			& 320  &  1.326E-03  &  0.997  &  3.307E-05  &  1.968  &  4.870E-07  &  2.946  \\
			& 640  &  6.637E-04  &  0.998  &  8.361E-06  &  1.984  &  6.206E-08  &  2.972  \\\hline		 	 				
			\multirow{5}{*}{1}
			&  40  &  2.043E-02  &     --      &  6.088E-03  &     --      &  1.117E-03   &     --      \\
			&  80  &  1.047E-02  &  0.964  &  1.822E-03  &  1.741  &  1.924E-04  &  2.537  \\
			& 160  &  5.272E-03  &  0.990  &  4.955E-04  &  1.878  &  2.788E-05  &  2.787  \\
			& 320  &  2.646E-03  &  0.995  &  1.293E-04  &  1.938  &  3.752E-06  &  2.893  \\
			& 640  &  1.326E-03  &  0.997  &  3.307E-05  &  1.968  &  4.869E-07  &  2.946  \\\hline 	  
			\multirow{5}{*}{2}
			&  40  &  3.941E-02  &     --      &  1.747E-02  &     --      &  4.522E-03  &     --      \\
			&  80  &  2.045E-02  &  0.946  &  6.098E-03  &  1.518  &  1.118E-03  &  2.016  \\
			& 160  &  1.047E-02  &  0.966  &  1.822E-03  &  1.743  &  1.924E-04  &  2.539  \\
			& 320  &  5.273E-03  &  0.990  &  4.955E-04  &  1.878  &  2.788E-05  &  2.787  \\
			& 640  &  2.646E-03  &  0.995  &  1.293E-04  &  1.938  &  3.752E-06  &  2.894 \\\hline	  
		\end{tabular}
	\end{small}
\end{table}

\vspace{0.5cm}

\textbf{Example 2.}
We test the porous medium equation (PME) \cite{muskat1937flow, aronson1986porous}
\begin{align}
\label{eq:PME}
u_{t}=(u^m)_{xx},
\end{align}
for some $m>1$. This equation describes a gas flowing isentropically in a porous medium, where the quantity $u$ represents density of the gas considered. However, for the PME, the classical solutions may not exist in general, even if the initial solution is smooth. Therefore, weak solutions must be considered, and their existence and uniqueness are studied in \cite{alt1983quasilinear, duyn1982nonstationary, otto1996l1}.

One famous weak solution of PME is the Barenblatt solution \cite{zel1950towards,barenblatt1952self}, which is defined as
$$B_{m}(x,t)=t^{-p}\big[ (1-\frac{p(m-1)}{2m}\frac{|x|^2}{t^{2p}})_{+} \big]^{1/(m-1)}, \ \ \ m>1,$$
where $u_{+}=\max(u,0)$ and $p=(m+1)^{-1}$. For any time $t>0$, the solution has a compact support $[-a_{m}(t),a_m(t)]$ with
$$a_{m}(t)=t^p\sqrt{\frac{2m}{p(m-1)}}.$$
Here, we choose $t=1$ as the initial time and the computation domain $[-6,6]$ with a zero boundary condition $u(\pm6,t)=0$. We plot the numerical solutions and exact solutions at $T=2$ with $200$ grid points, respectively, with $m=2, 3, 5$ and $8$ (Firgue \ref{Fig4}). Here, we only plot the results for the third order scheme, i.e. $k=3$, and $\beta$ is taken as 0.8. It is observed that our scheme is able to approximate the Barenblatt solution accurately without noticeable oscillations even with a large CFL number. On the other hand, the method with a smaller CFL generates sharper interface transition around $|x|=a_{m}$.

\begin{figure}
	\centering
	\subfigure[$m=2$. ]{
		\includegraphics[width=0.35\textwidth]{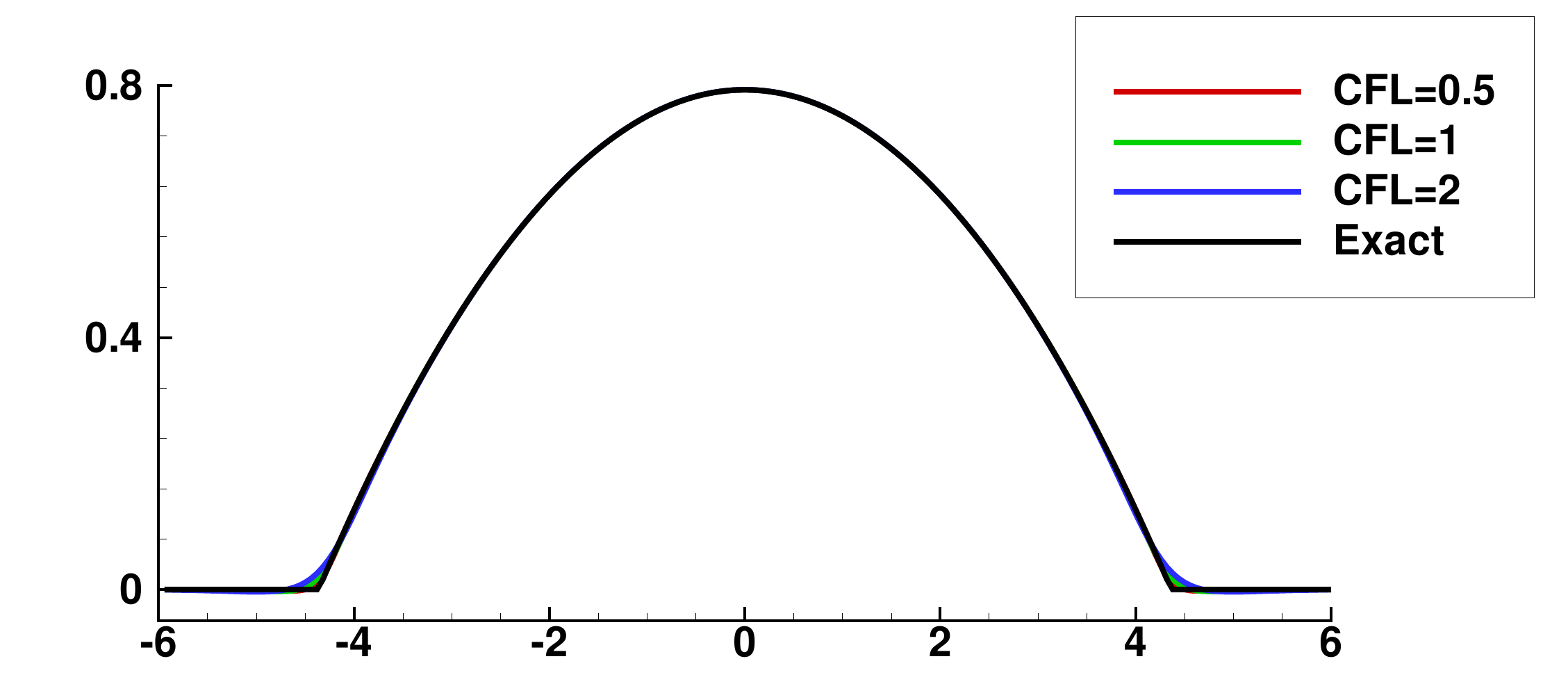}\label{Fig4.5}}
	\subfigure[$m=3$. ]{
		\includegraphics[width=0.35\textwidth]{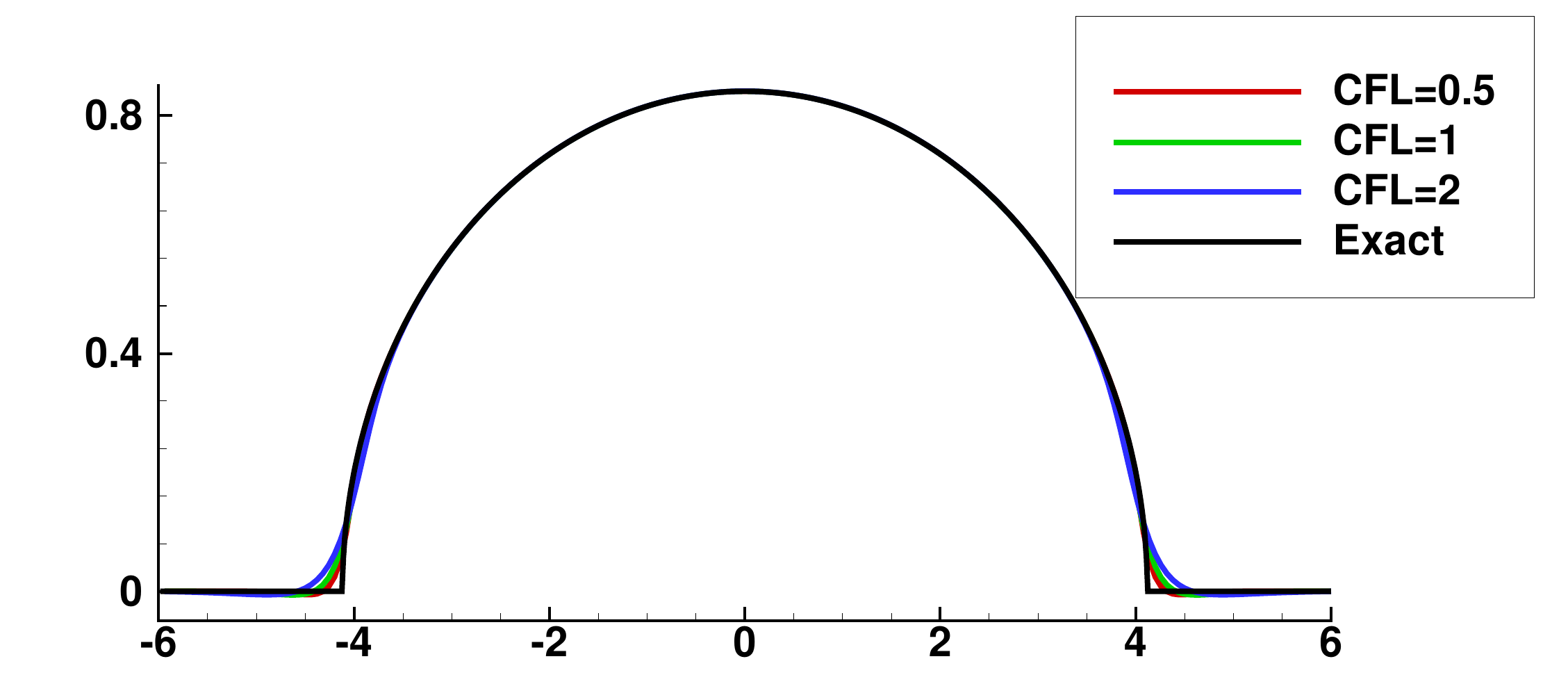}\label{Fig4.6}}
	\subfigure[$m=5$. ]{
		\includegraphics[width=0.35\textwidth]{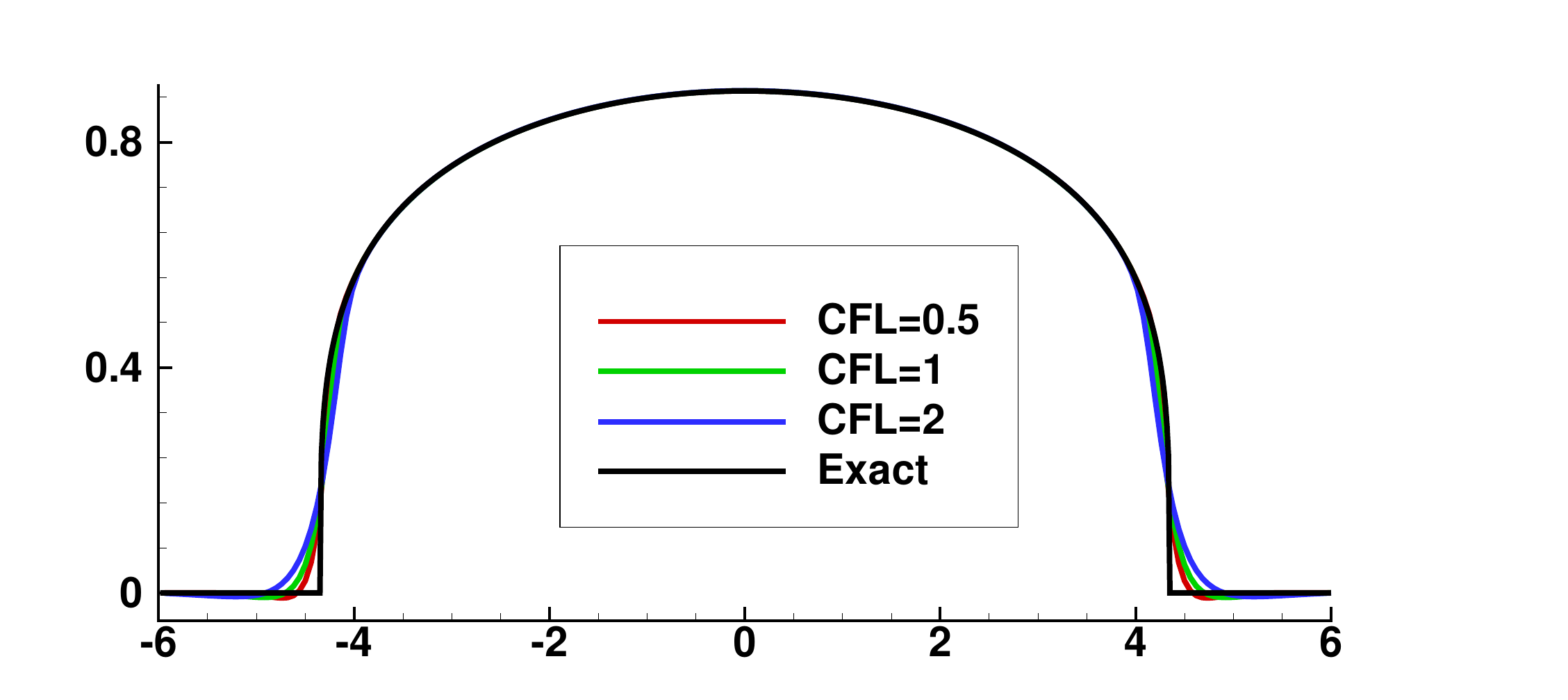}\label{Fig4.11}}
	\subfigure[$m=8$. ]{
		\includegraphics[width=0.35\textwidth]{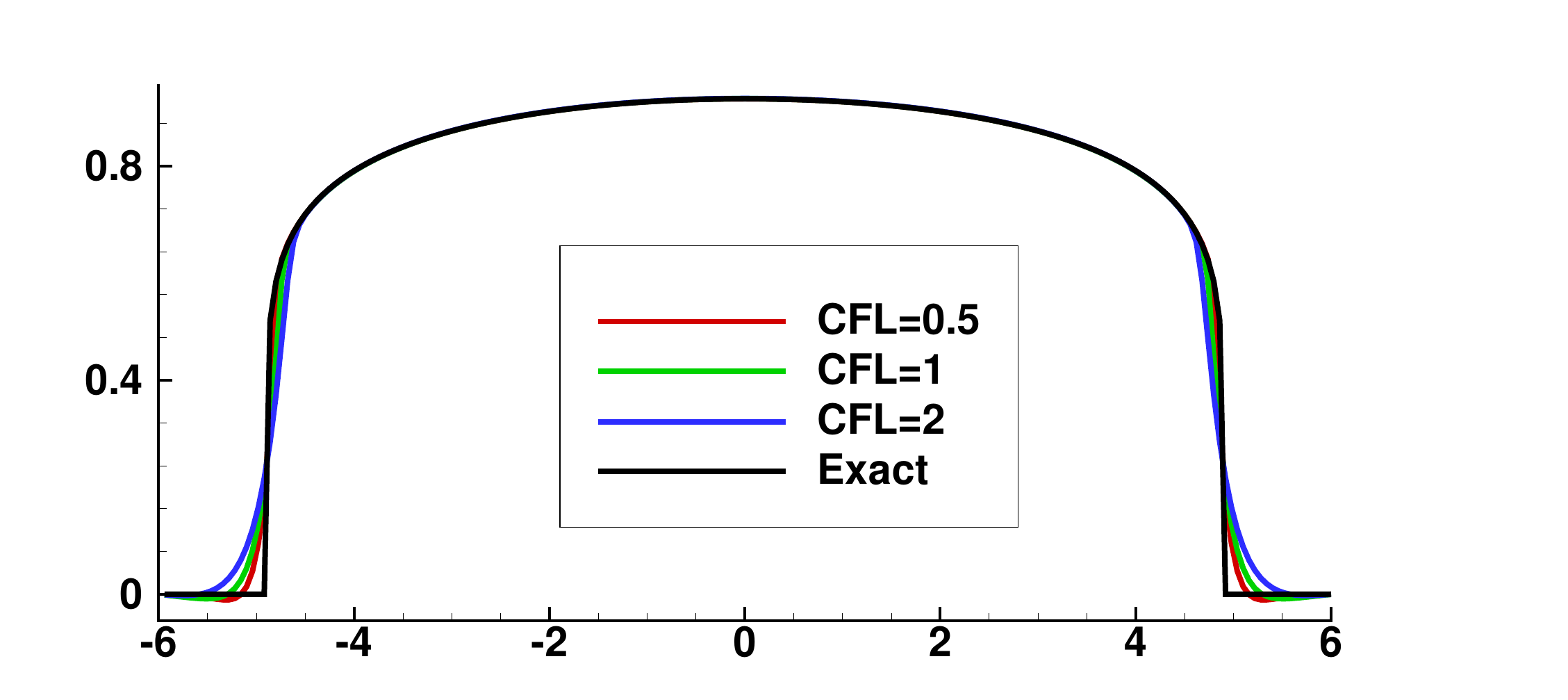}\label{Fig4.12}}
	\caption{\em Example 2: Barenblatt solution for PME. $N=200$ grid points. $k=3$. $\beta=0.8$. }
	\label{Fig4}
\end{figure}

\vspace{0.5cm}

\textbf{Example 3.}
Now, we consider the interaction of two boxes for PME \eqref{eq:PME}. Such a model describes how temperature changes when two hot spots are suddenly placed in the domain. Here, we choose the initial condition as
\begin{align}
u(x,0)=\left\{\begin{array}{ll}
1, & x\in(-4,-1),\\
2, & x\in(0,3),\\
0, & otherwise,
\end{array}
\right.
\end{align}
in which the two boxes have the different heights. We let $m=6$ in \eqref{eq:PME}. In Figure \ref{Fig11}, we show the numerical solution at several instances of time for $k=3$. The computational domain is chosen as $[-6,6]$ and a zero boundary condition $u(\pm 6,t)=0$ is imposed. Here, we use $N=400$ grid points. 
Note that the exact solution is unknown and we benchmark the scheme against the reference. 
 It is observed that the scheme is able to capture the sharp interface even though a large CFL number is used and numerical solutions agree with the reference solution very well.

\begin{figure}
	\centering
	\subfigure[$t=0.02$.]{
		\includegraphics[width=0.4\textwidth]{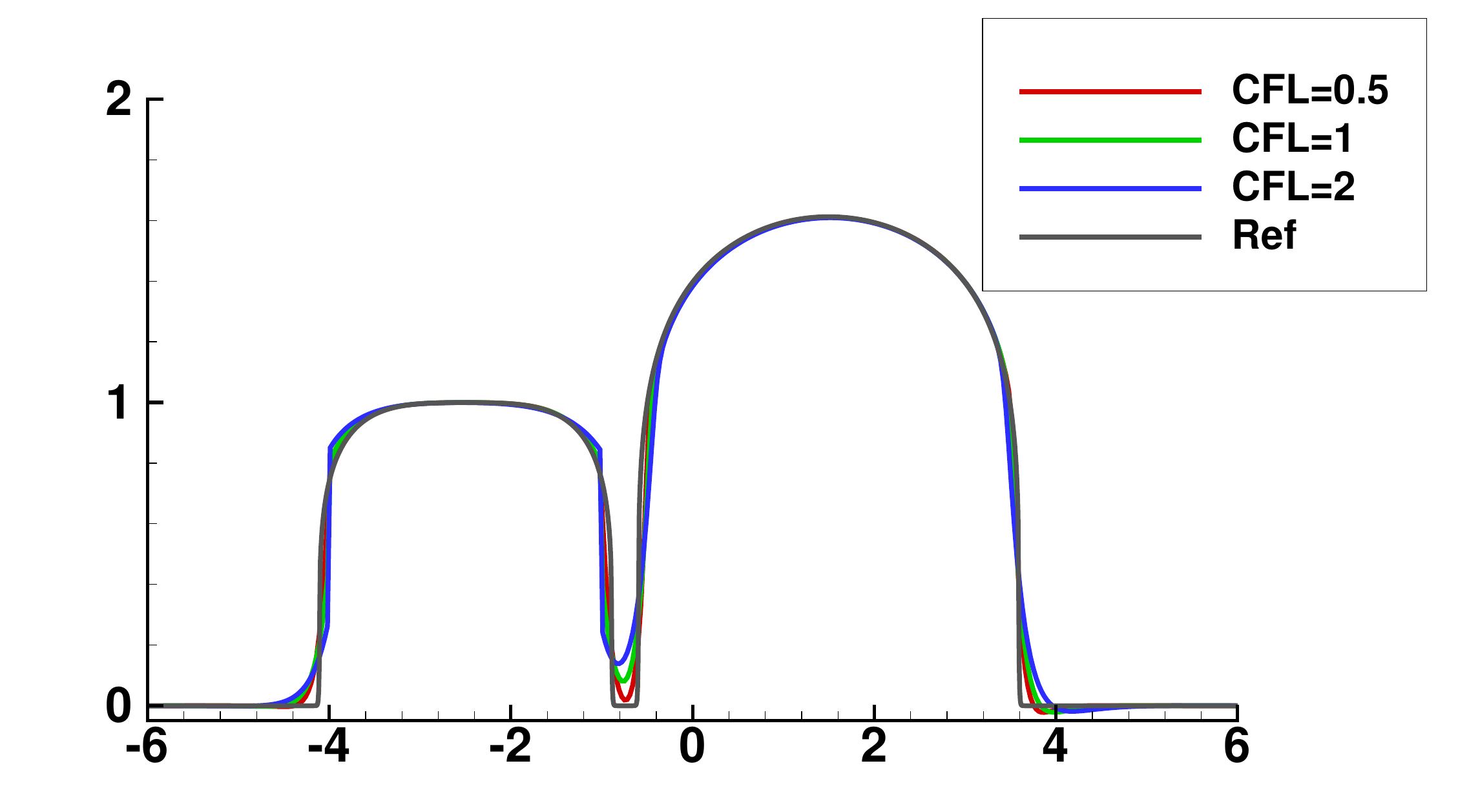}\label{Fig11.5}}
	\subfigure[$t=0.04$.]{
		\includegraphics[width=0.4\textwidth]{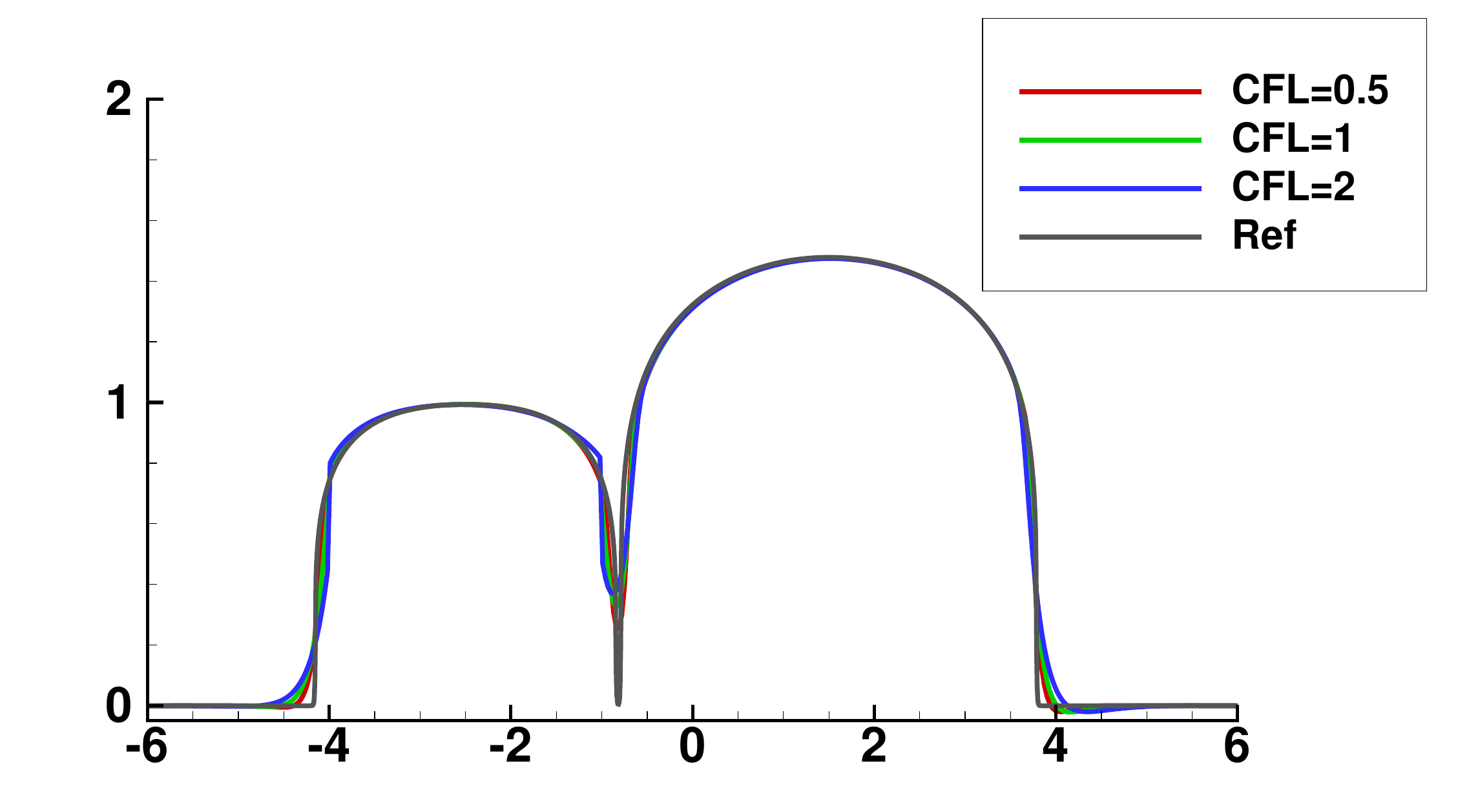}\label{Fig11.6}}
	\subfigure[$t=0.06$.]{
		\includegraphics[width=0.4\textwidth]{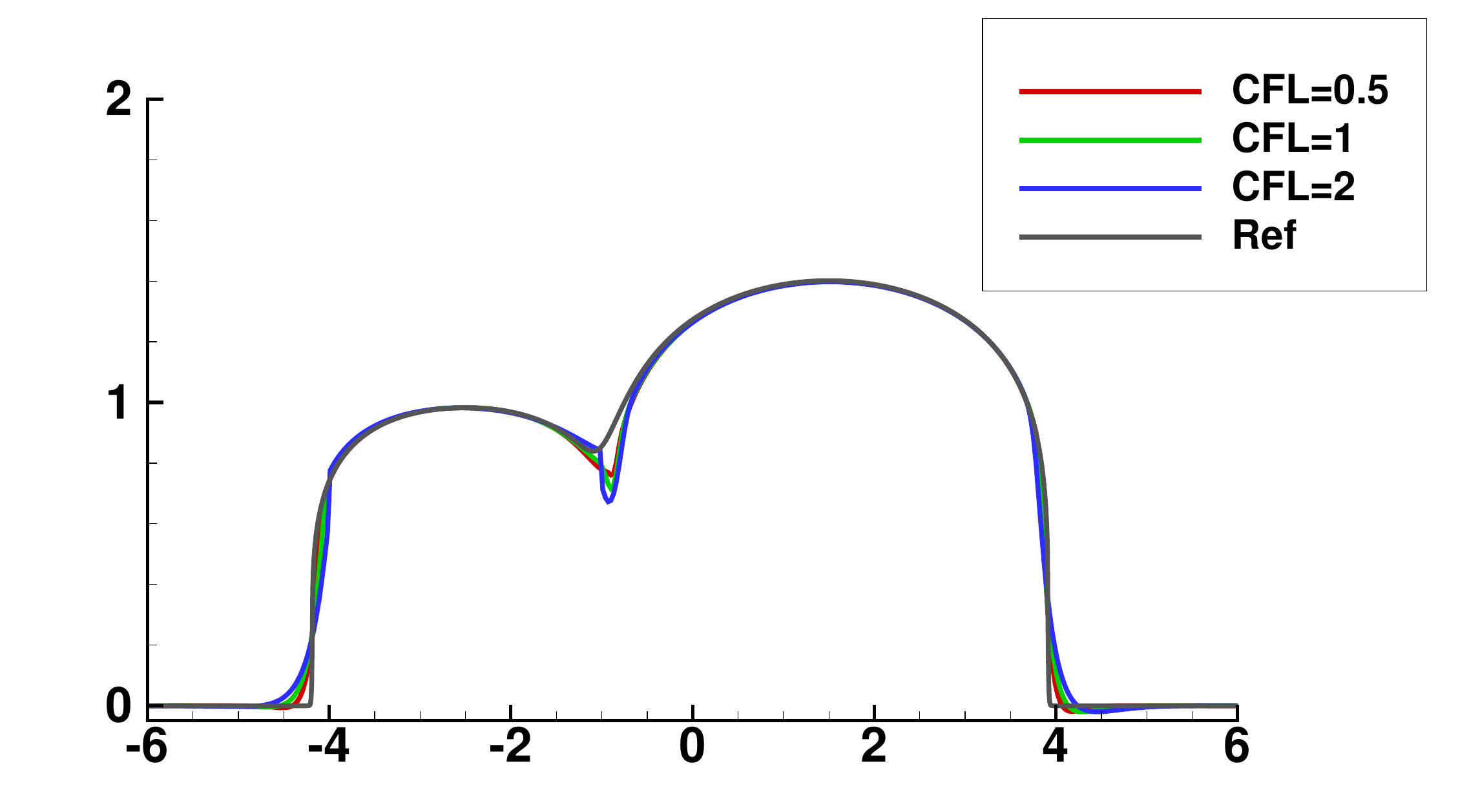}\label{Fig11.11}}
	\subfigure[$t=0.12$.]{
		\includegraphics[width=0.4\textwidth]{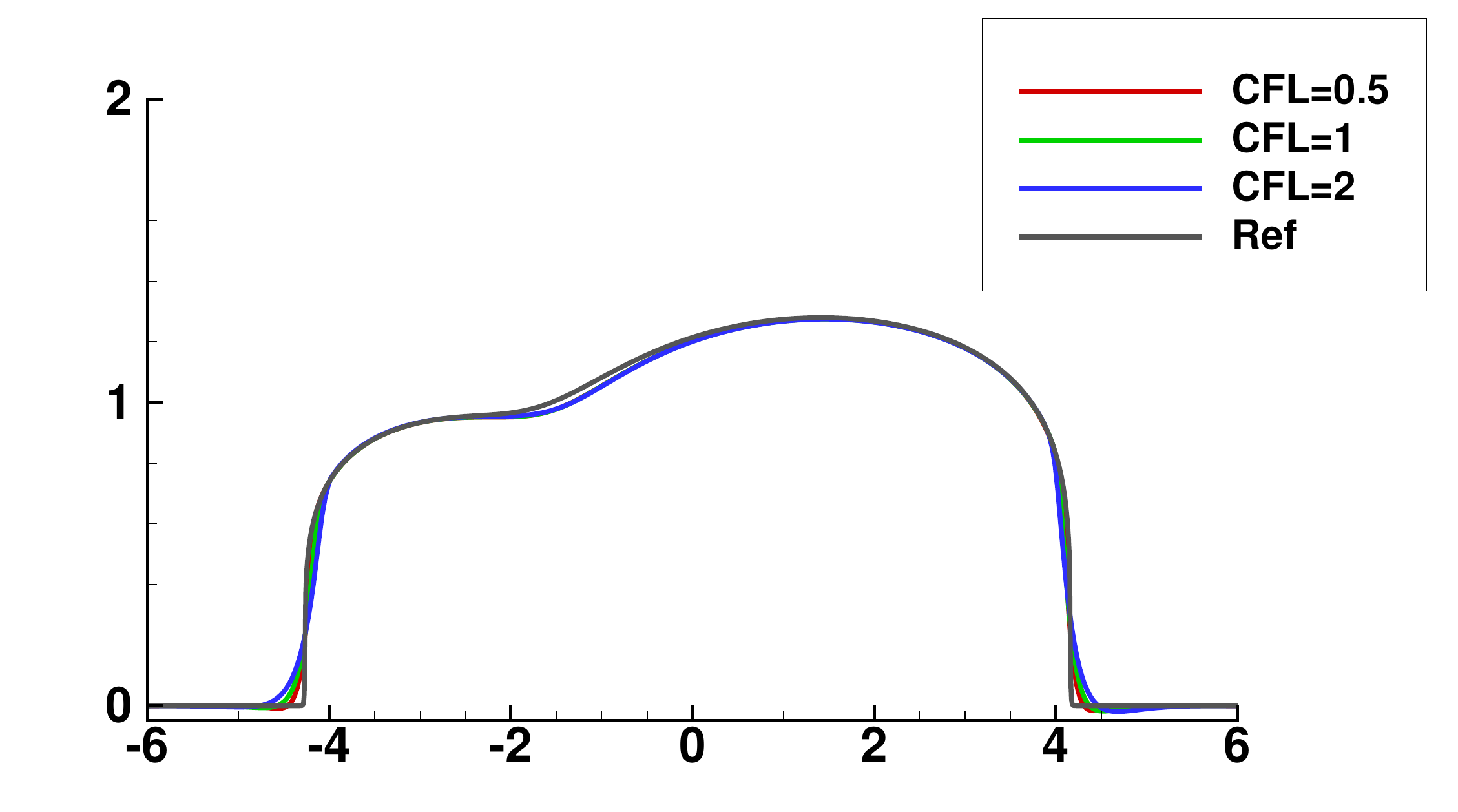}\label{Fig11.12}}
	\caption{\em Example 3: Interaction of the two-Box solution with different heights. $N=400$ grid points. $k=3$. $\beta=0.8$. }
	\label{Fig11}
\end{figure}

\vspace{0.5cm}

\textbf{Example 4.}
Next, let us consider the Buckley-Leverett equation \cite{buckley1942mechanism}
\begin{align}
\label{eq:BL}
u_{t}+f(u)_x=\epsilon(\nu(u)u_{x})_{x}.
\end{align}
In fluid dynamics, this equation is used to model two-phase flow in porous media, such as displacing oil by water in a one-dimensional or quasi-one-dimensional reservoir.
We choose 
\begin{align}
\label{eq:BLdiff}
\nu(u)=\left\{\begin{array}{ll}
4u(1-u), & 0\leq u\leq 1,\\
0, & \text{otherwise,}\\
\end{array}
\right.
\end{align}
and consider the flux without gravitational effects
\begin{align}
\label{eq:withoutgrav}
f(u)=\frac{u^2}{u^2+(1-u)^2},
\end{align}
as well as with gravitational effects
\begin{align}
\label{eq:withgrav}
f(u)=\frac{u^2}{u^2+(1-u)^2}(1-5(1-u)^2).
\end{align}
In the simulation, we let $\epsilon=0.01$.
The initial condition
\begin{align*}
u(x,0)=\left\{\begin{array}{ll}
0, & 0\leq x<1-\frac{1}{\sqrt{2}},\\
1, & 1-\frac{1}{\sqrt{2}}\leq x\leq1.\\
\end{array}
\right.
\end{align*}

Numerical solutions for $k=1,\,2,\,3$ are presented in Figure \ref{Fig2:bl}. It is observed that the scheme with a small CFL, e.g. 0.5 outperforms the one with a large CFL, e.g. 2. In fact, even though the scheme is unconditionally stable, the performance may not be satisfactory when a exceedingly large CFL number is used, which will introduce too much numerical diffusion and thus smear the interface. Nevertheless,
when a small CFL number is used, the scheme is able to solve both Riemann problems accurately without generating noticeable spurious oscillations, and the solution is benchmarked against the results reported in \cite{kurganov2000new, liu2011high}.

\begin{figure}
	\centering
	\subfigure[$k=1$. $\beta=1$.]{
		\includegraphics[width=0.3\textwidth]{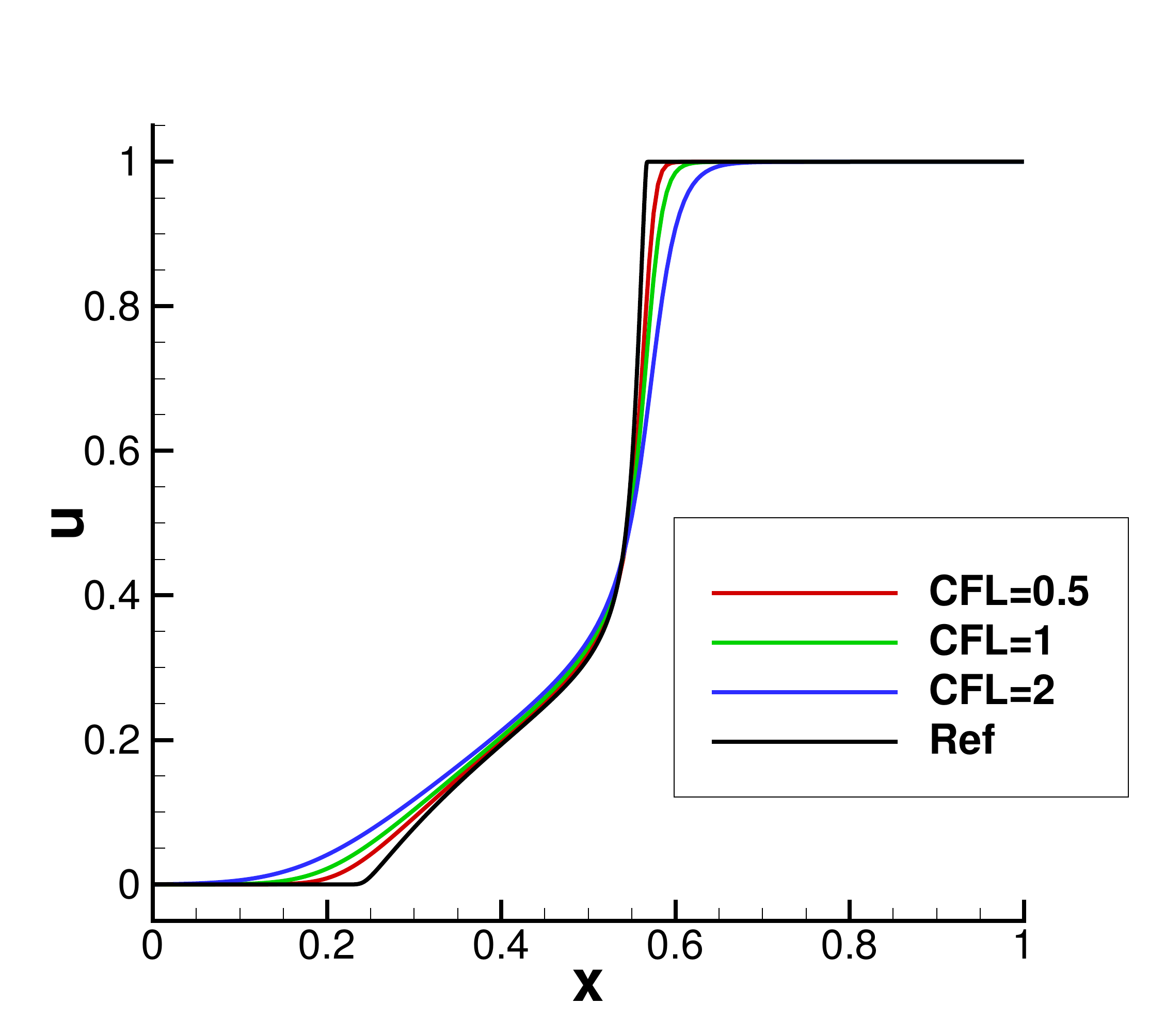}\label{Fig2.1:bl}}
	\subfigure[$k=2$. $\beta=0.5$.]{
		\includegraphics[width=0.3\textwidth]{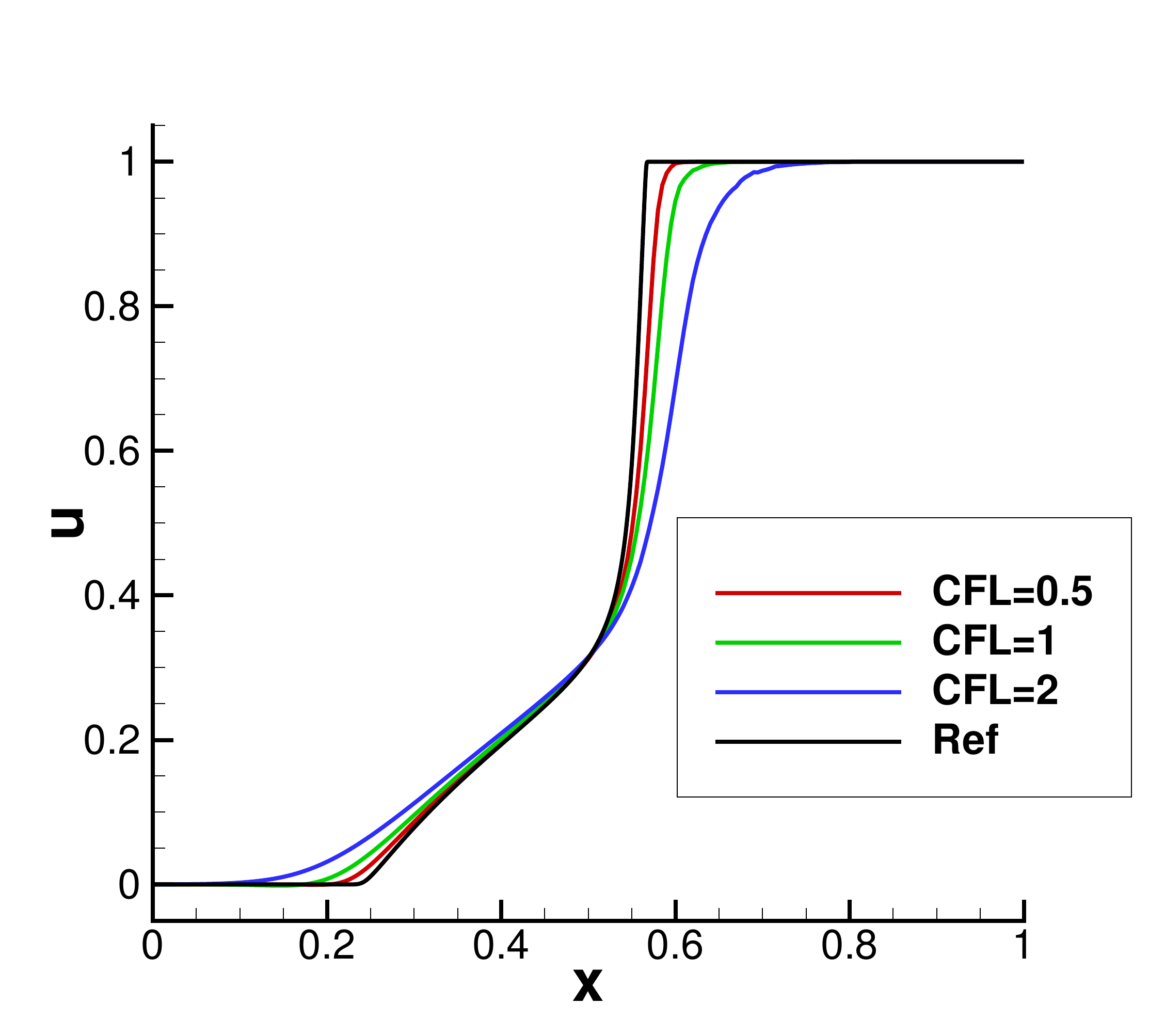}\label{Fig2.2:bl}}
	\subfigure[$k=3$. $\beta=0.4$.]{
		\includegraphics[width=0.3\textwidth]{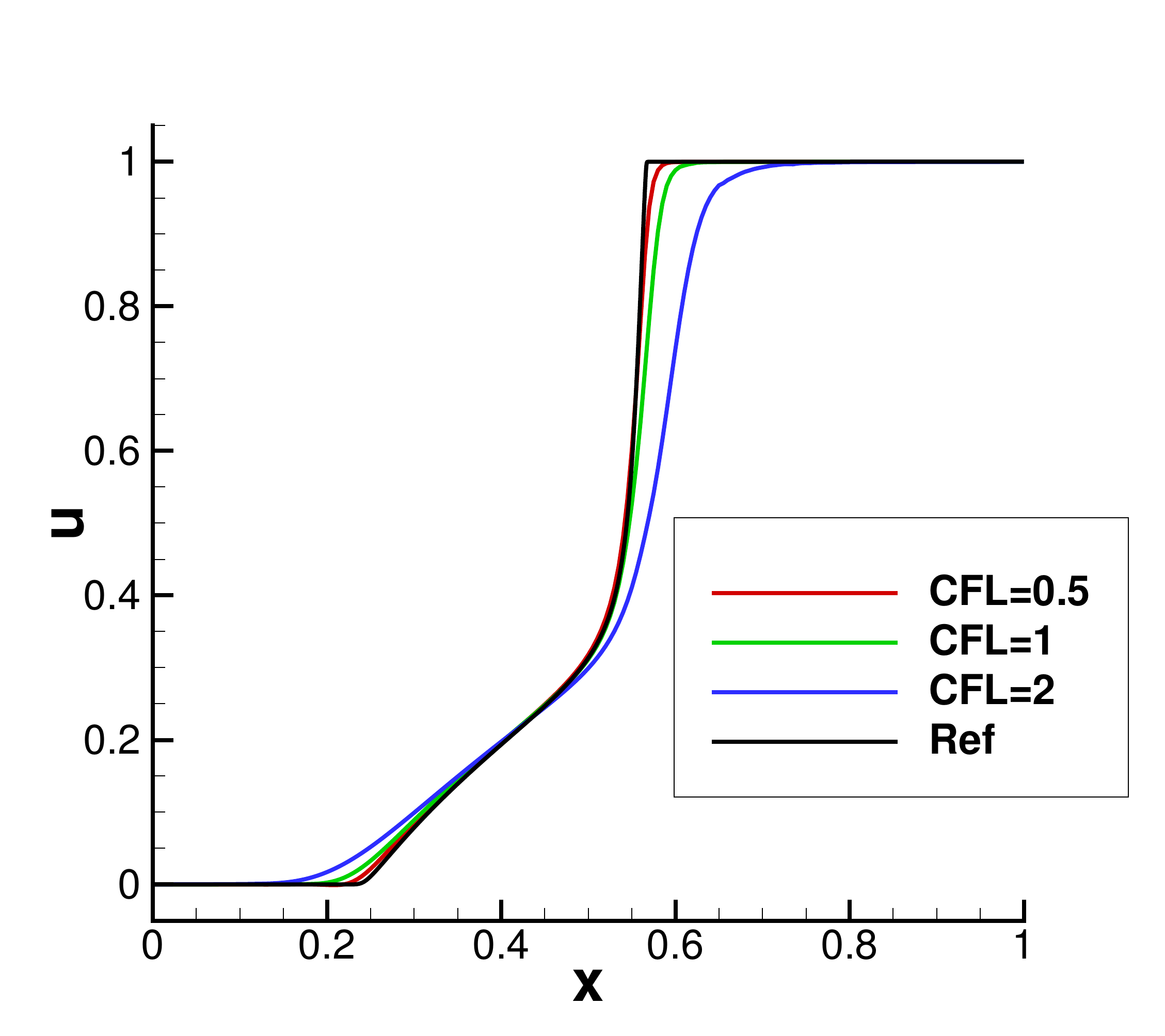}\label{Fig2.3:bl}}
	\subfigure[$k=1$. $\beta=1$.]{
		\includegraphics[width=0.3\textwidth]{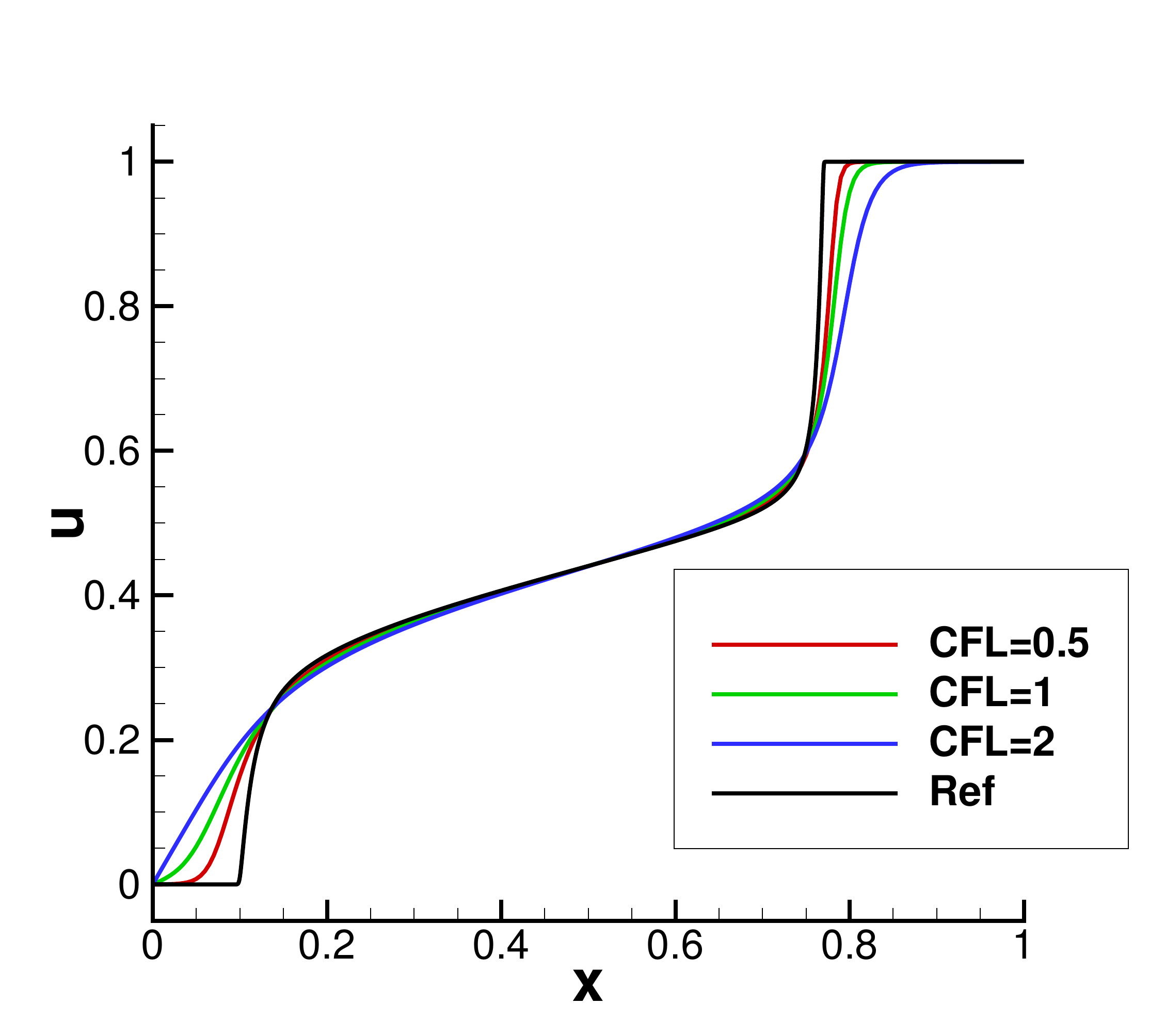}\label{Fig2.4:bl}}
	\subfigure[$k=2$. $\beta=0.5$.]{
		\includegraphics[width=0.3\textwidth]{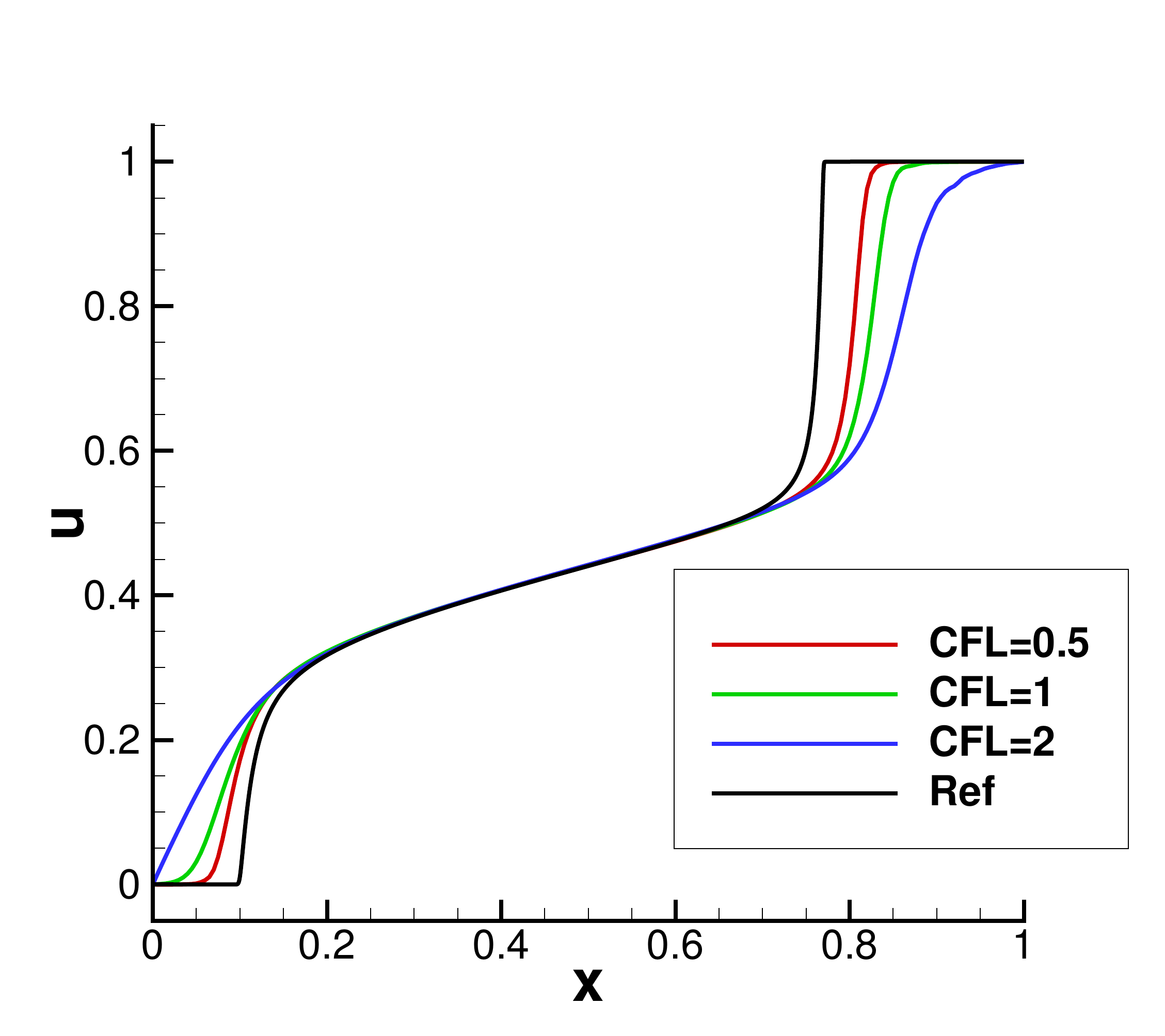}\label{Fig2.5:bl}}
	\subfigure[$k=3$. $\beta=0.4$.]{
		\includegraphics[width=0.3\textwidth]{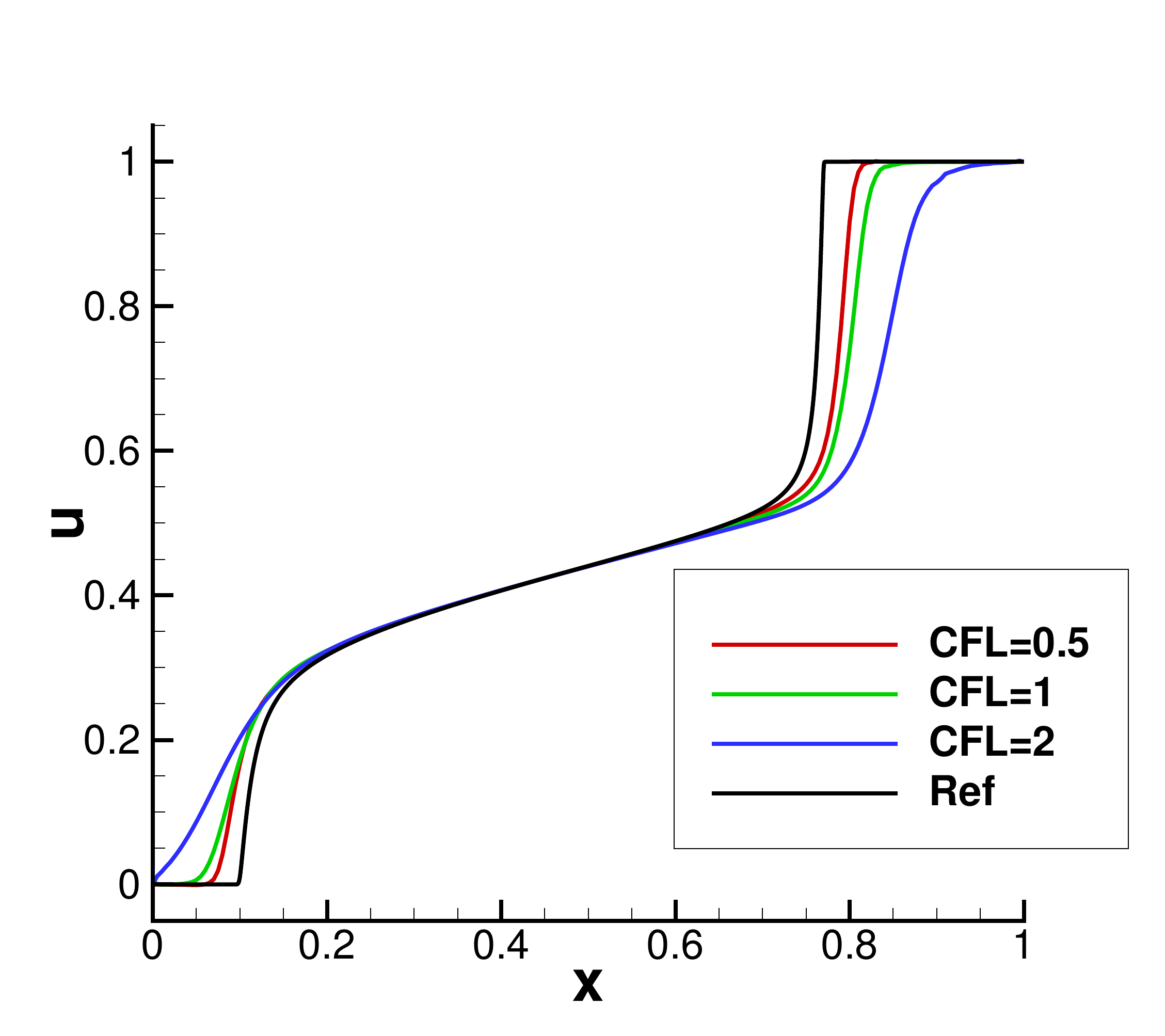}\label{Fig2.6:bl}}
	\caption{\em Example 4: Buckley-Leverett equation. $T=0.2$. $N=200$ grid points. First row: without gravitition; second row: with gravitition.}
	\label{Fig2:bl}
\end{figure}

\vspace{0.5cm}

\textbf{Example 5.}
In this example, we consider a strongly degenerate parabolic convection-diffusion equation
\begin{align}
\label{eq:SD}
u_{t}+f(u)_x=\epsilon(\nu(u)u_x)_x
\end{align}
We take $\epsilon=0.1$, $f(u)=u^2$, and
\begin{align}
\nu(u)=\left\{\begin{array}{ll}
0, & |u|\leq 0.25,\\
1, & |u|>0.25 .\\
\end{array}
\right.
\end{align}
The choice of $\mu$ will lead to an interesting fact that the equation is hyperbolic when $u\in[-0.25,0.25]$ and parabolic elsewhere. We solve the problem with the initial function
\begin{align}
u(x,0)=\left\{\begin{array}{ll}
1, & -\frac{1}{\sqrt{2}}-0.4<x<-\frac{1}{\sqrt{2}}+0.4,\\
-1,& \frac{1}{\sqrt{2}}-0.4<x<\frac{1}{\sqrt{2}}+0.4.\\
0, & otherwise\\
\end{array}
\right.
\end{align}
 Numerical results are presented in Figure \ref{Fig8}. In particular, we compare the performance of the schemes with different orders of accuracy, i.e, $k=1
 ,2\,3$. It is observed that the high order scheme performs better in capturing the sharp interface as well as the kinks where the equation changes its type. As expected,  the performance deteriorates 
 when a large CFL number is used. 

\begin{figure}
	\centering
	\subfigure[ $k=1$. $\beta=1$.]{
		\includegraphics[width=0.3\textwidth]{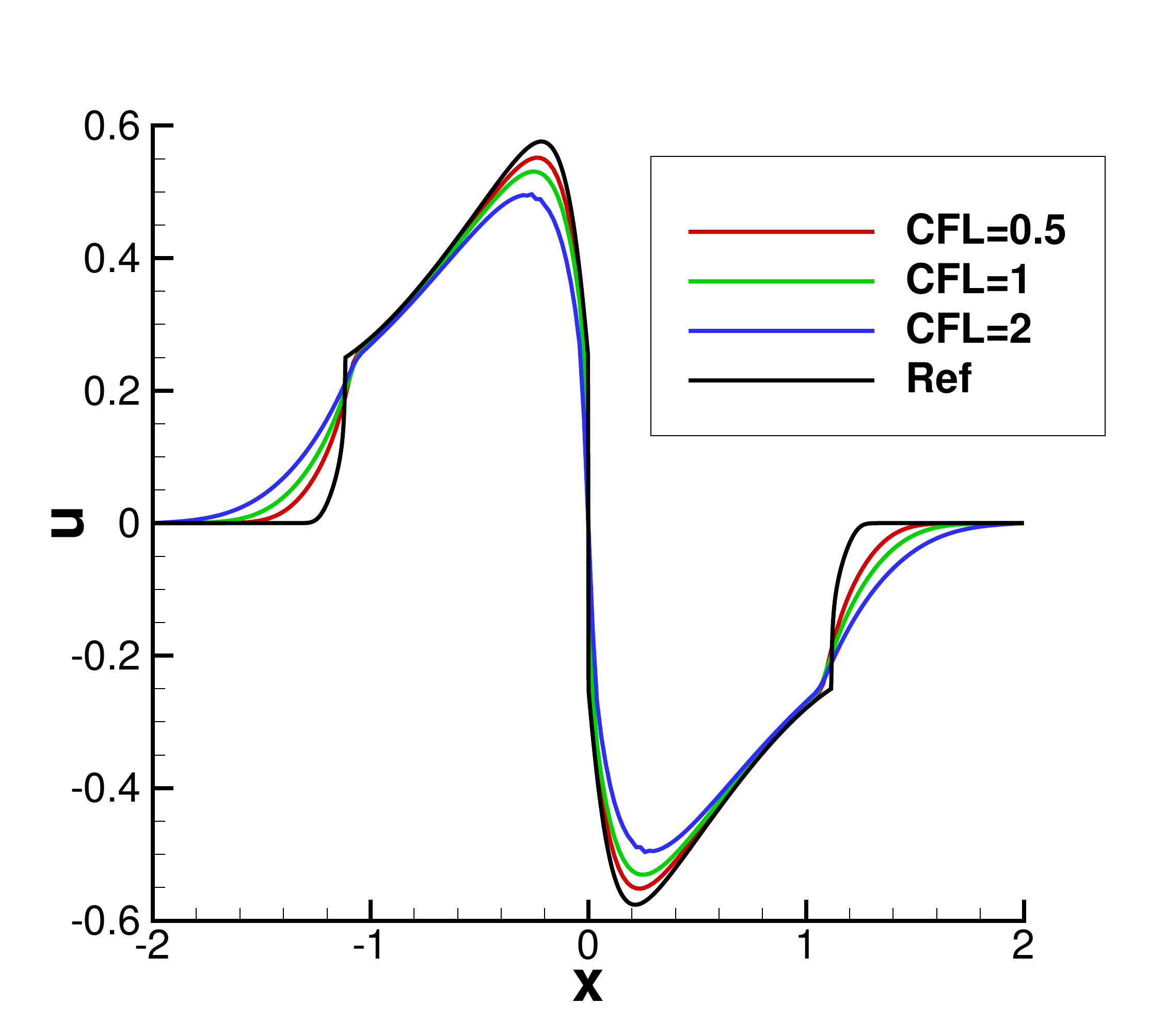}\label{Fig8.1}}
	\subfigure[$k=2$. $\beta=0.5$.]{
		\includegraphics[width=0.3\textwidth]{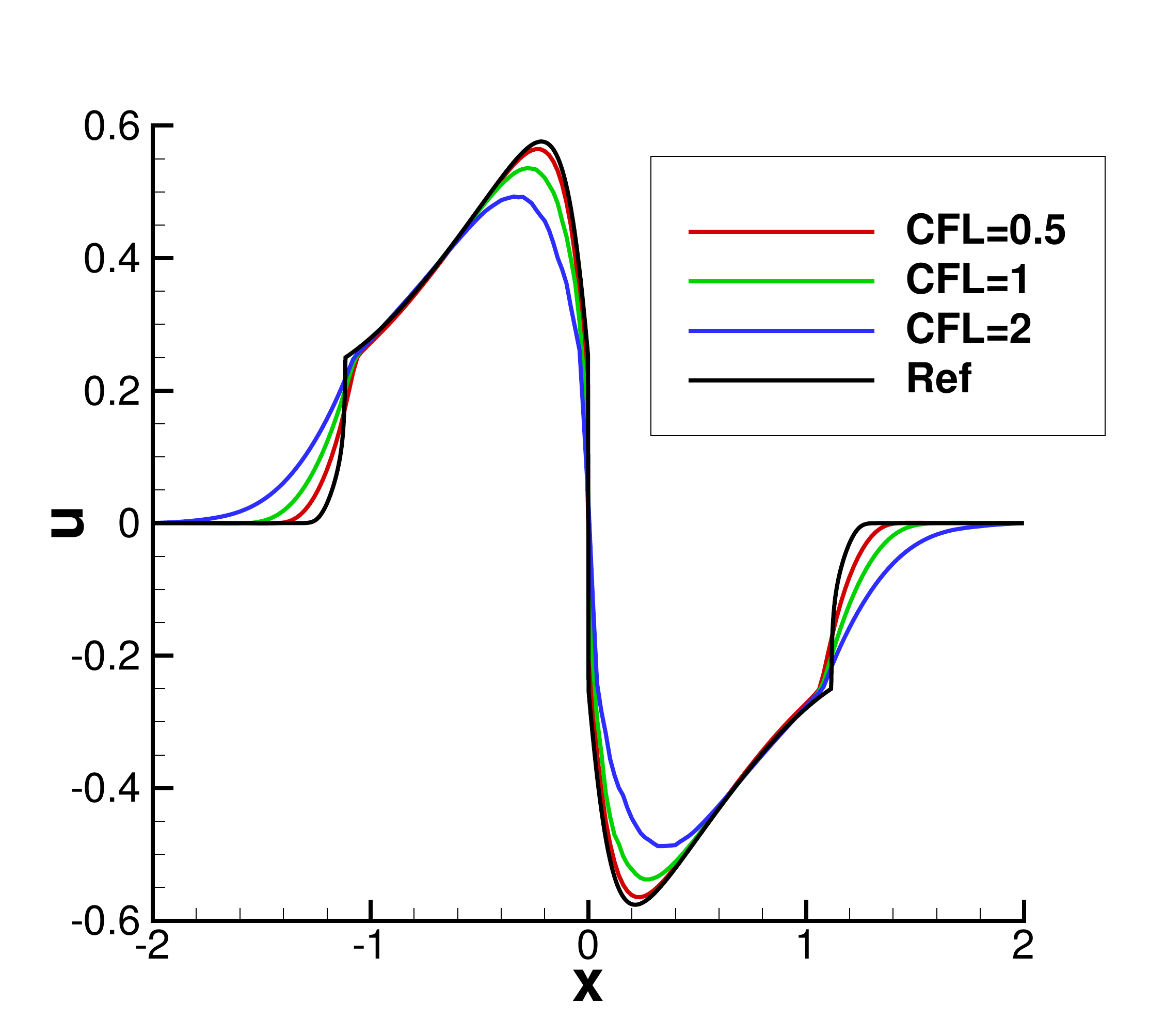}\label{Fig8.2}}
	\subfigure[$k=3$. $\beta=0.4$.]{
		\includegraphics[width=0.3\textwidth]{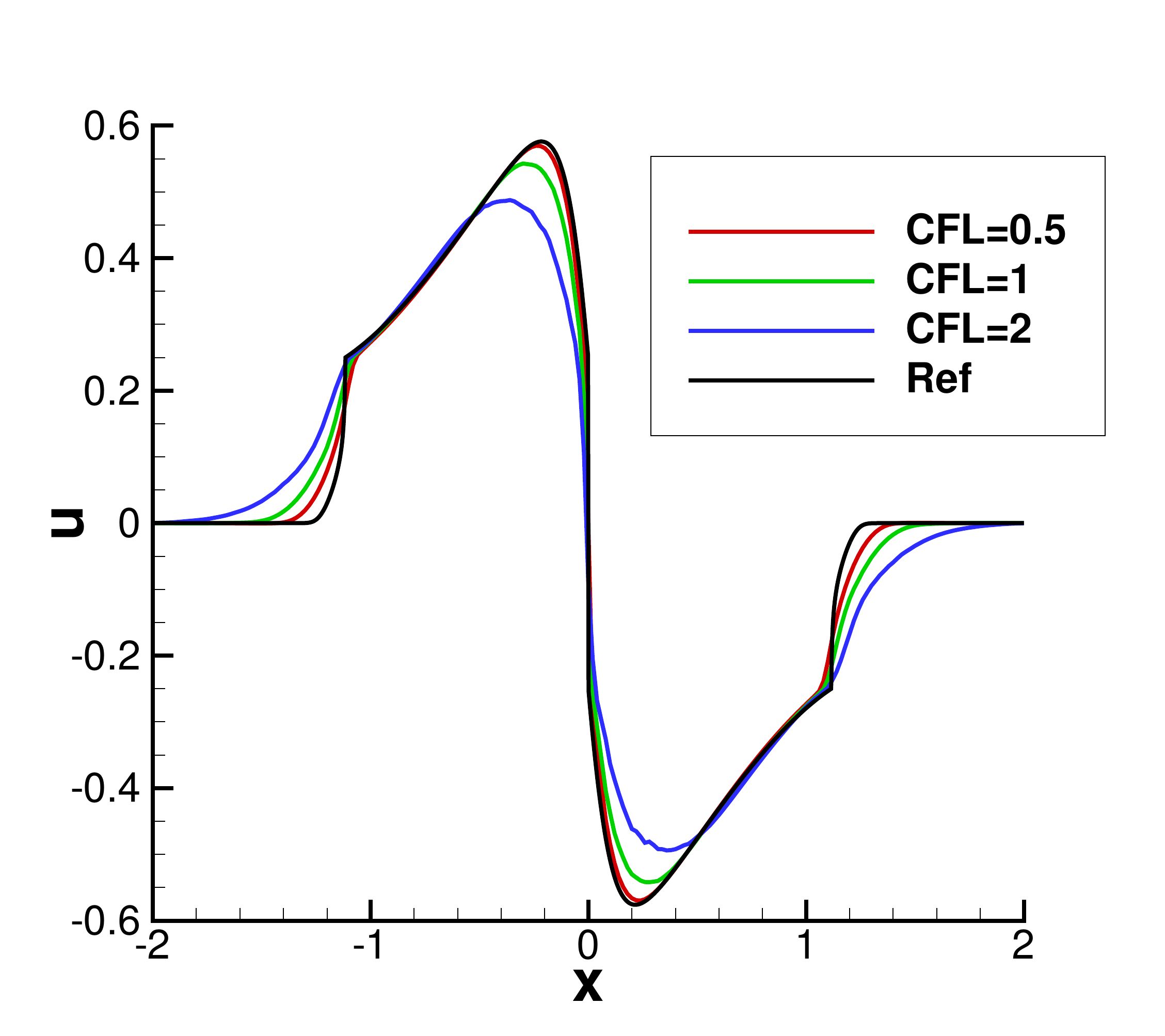}\label{Fig8.3}}
	\caption{\em Example 5: 1D strong degenerate parabolic equation. $T=0.7$. $N=200$ grid points. }
	\label{Fig8}
\end{figure}

\vspace{0.5cm}

\textbf{Example 6.}
We consider the two-dimensional strongly degenerate parabolic convection-diffusion equation
\begin{align}
\label{eq:SD2D}
u_{t}+f(u)_x+f(u)_y=\epsilon(\nu(u)u_x)_x+\epsilon(\nu(u)u_y)_y,
\end{align}
in which, $\epsilon$, $f(u)$ and $\nu(u)$ are the same as the one-dimensional case. The initial function is given as
\begin{align}
u(x,y,0)=\left\{\begin{array}{ll}
1, & (x+0.5)^2+(y+0.5)^2<0.16,\\
-1,& (x-0.5)^2+(y-0.5)^2<0.16,\\
0, & otherwise.\\
\end{array}
\right.
\end{align}
The solutions at $T=0.5$ computed by the third order scheme with $\beta=0.2$, $CFL=0.5$ and $200\times200$ grid points are shown in Figure \ref{Fig9}, which  agree well with results provided in \cite{liu2011high}.

\begin{figure}
	\centering
	\subfigure{
		\includegraphics[width=0.35\textwidth]{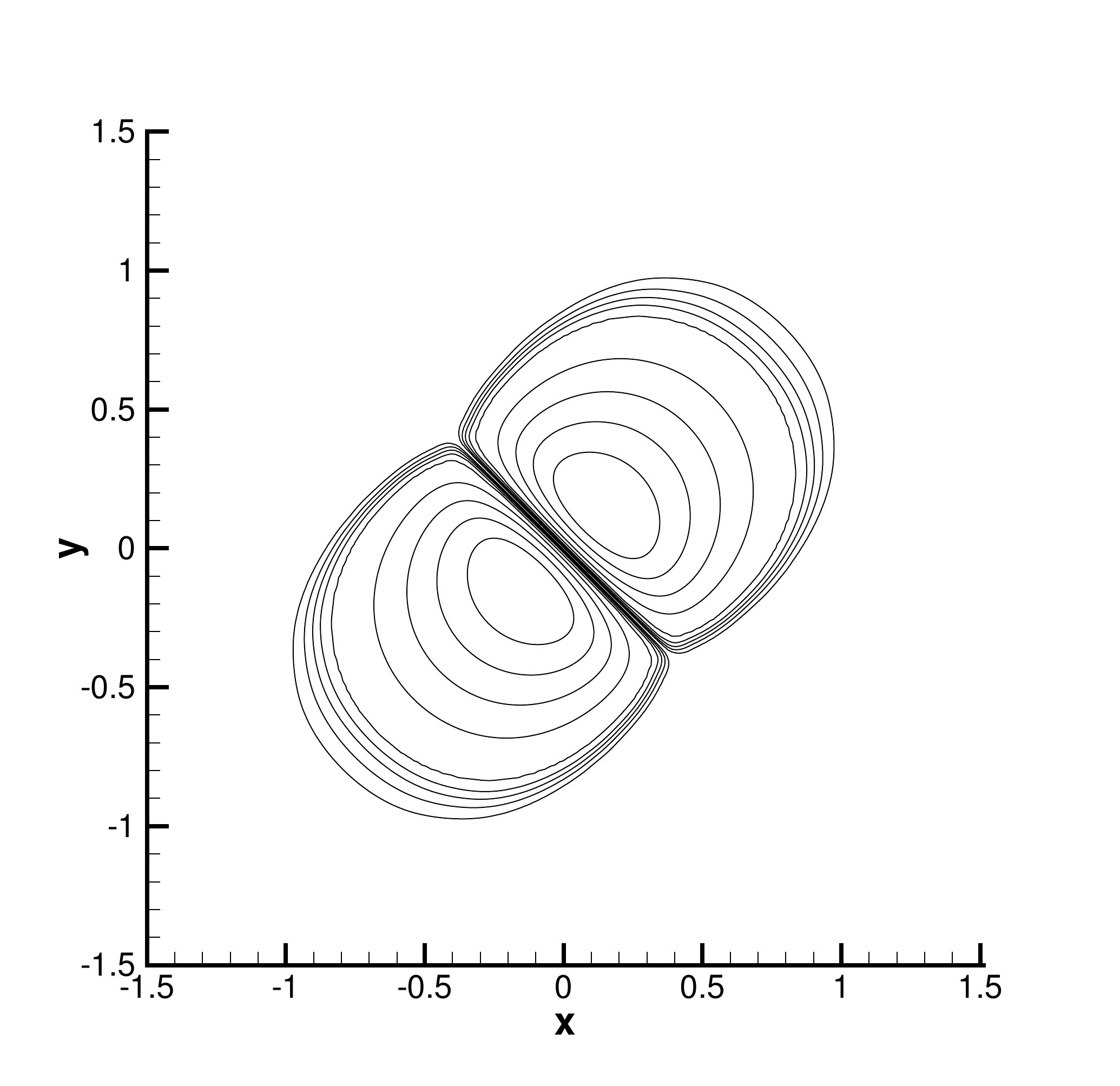}\label{Fig9.1}}
	\subfigure{
		\includegraphics[width=0.35\textwidth]{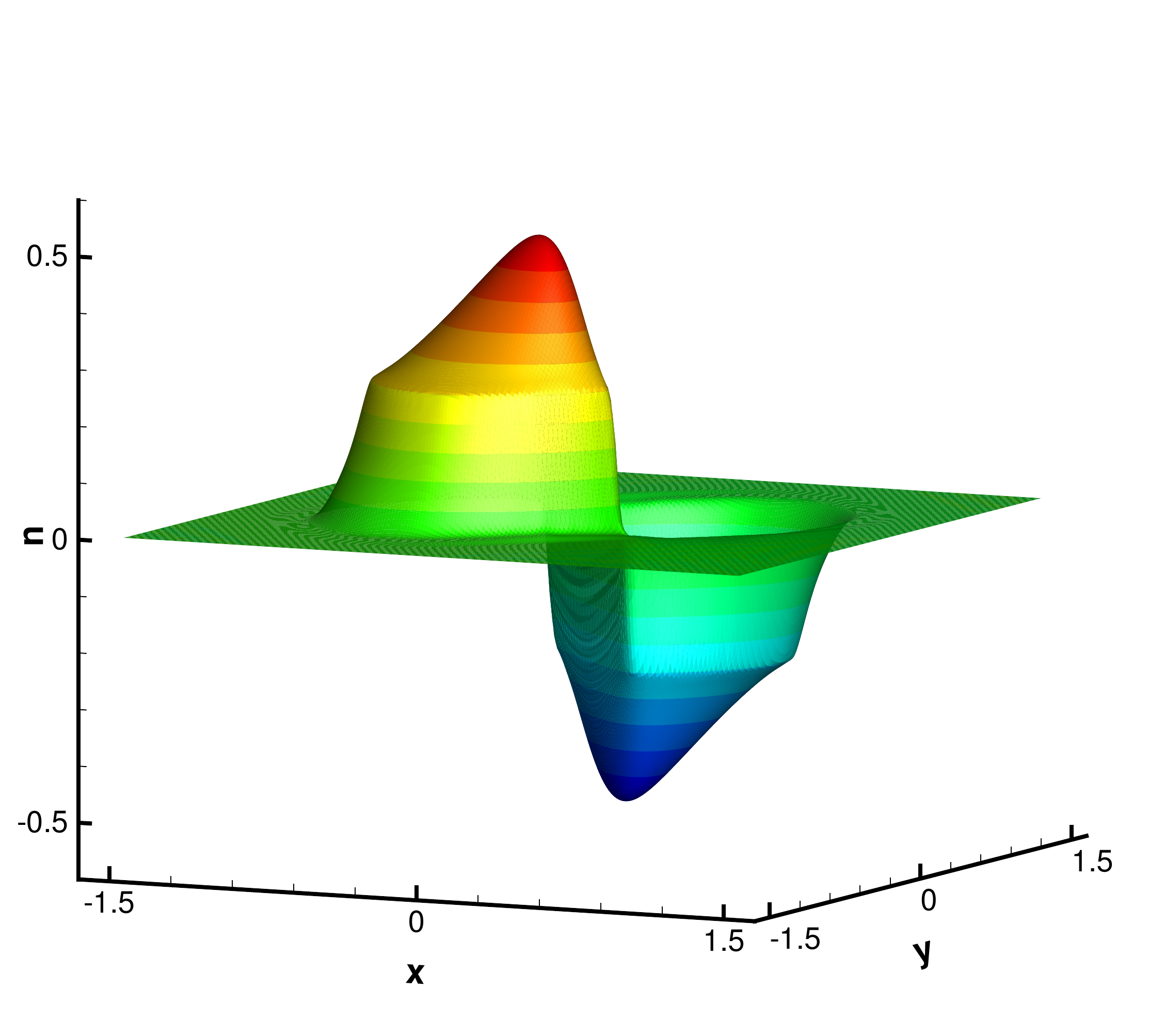}\label{Fig9.2}}
	\caption{\em Examplr 6: 2D strong degenerate parabolic equation. $k=3$. $T=0.5$. $CFL=0.5$. $200\times200$ grid points.}
	\label{Fig9}
\end{figure}

\vspace{0.5cm}

\textbf{Example 7.}
As the last example, we solve the two-dimensional Buckley-Leverett Equation
\begin{align}
\label{eq:BL2D}
u_{t}+f_{1}(u)_{x}+f_{2}(u)_{y}=\epsilon(u_{xx}+u_{yy}),
\end{align}
where $\epsilon=0.01$ and the flux functions are given as
$$f_{1}(u)=\frac{u^2}{u^2+(1-u)^2}, \ \ \ f_{2}(u)=(1-5(1-u)^2)f_{1}(u).$$
We compute the problem on $[-1.5,1.5]\times[-1.5,1.5]$, with the initial condition
\begin{align*}
u(x,y,0)=\left\{\begin{array}{ll}
1, & x^2+y^2<0.5,\\
0, & \text{otherwise}.\\
\end{array}
\right.
\end{align*}
Here, we only show the results computed by the third order scheme with $200\times200$ grid points in Figure \ref{Fig13}. The results agree with the those reported in  \cite{kurganov2000new}, demonstrating the effectiveness of the scheme for solving this challenging two-dimensional problem.

\begin{figure}
	\centering
	\subfigure{
		\includegraphics[width=0.35\textwidth]{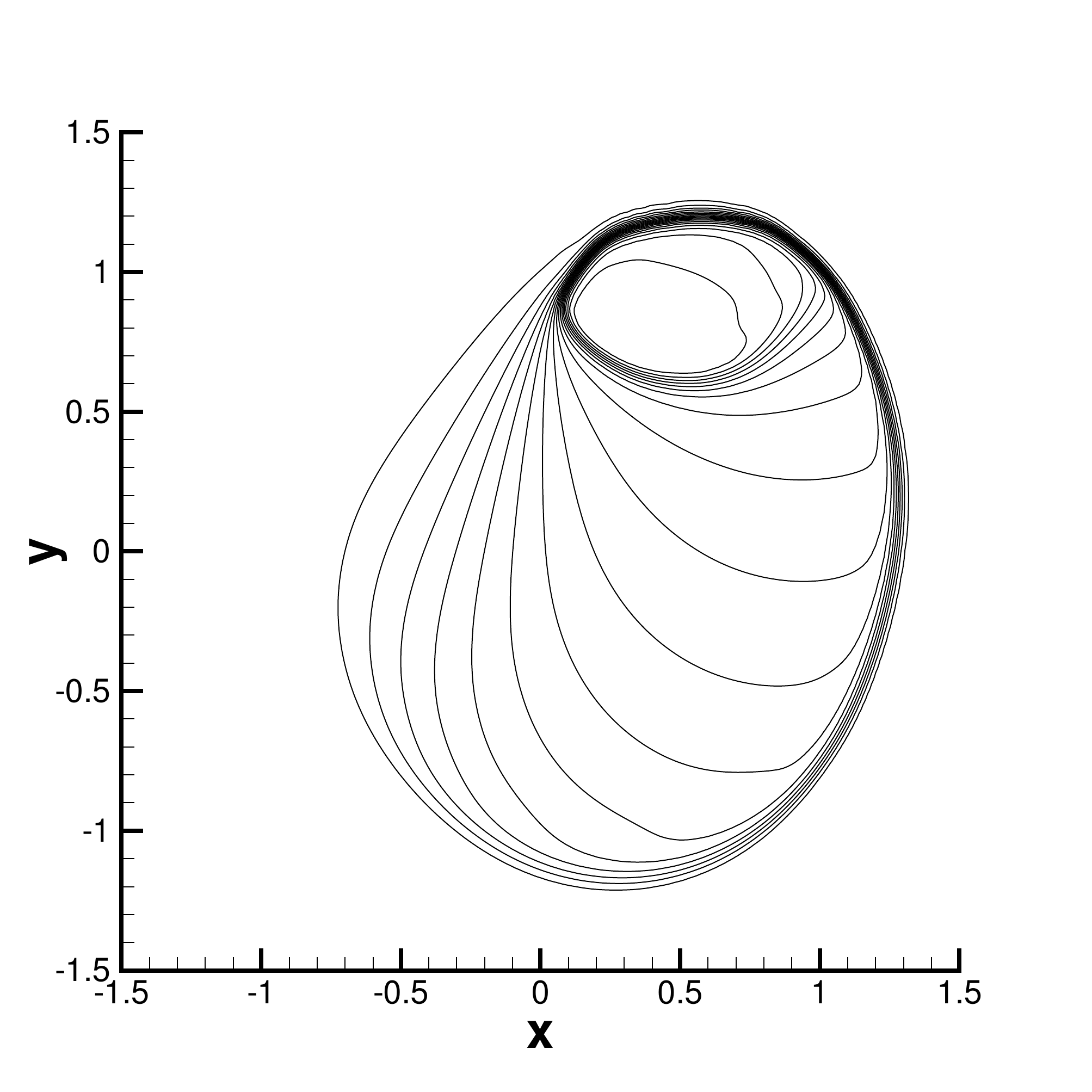}\label{Fig13.1}}
	\subfigure{
		\includegraphics[width=0.35\textwidth]{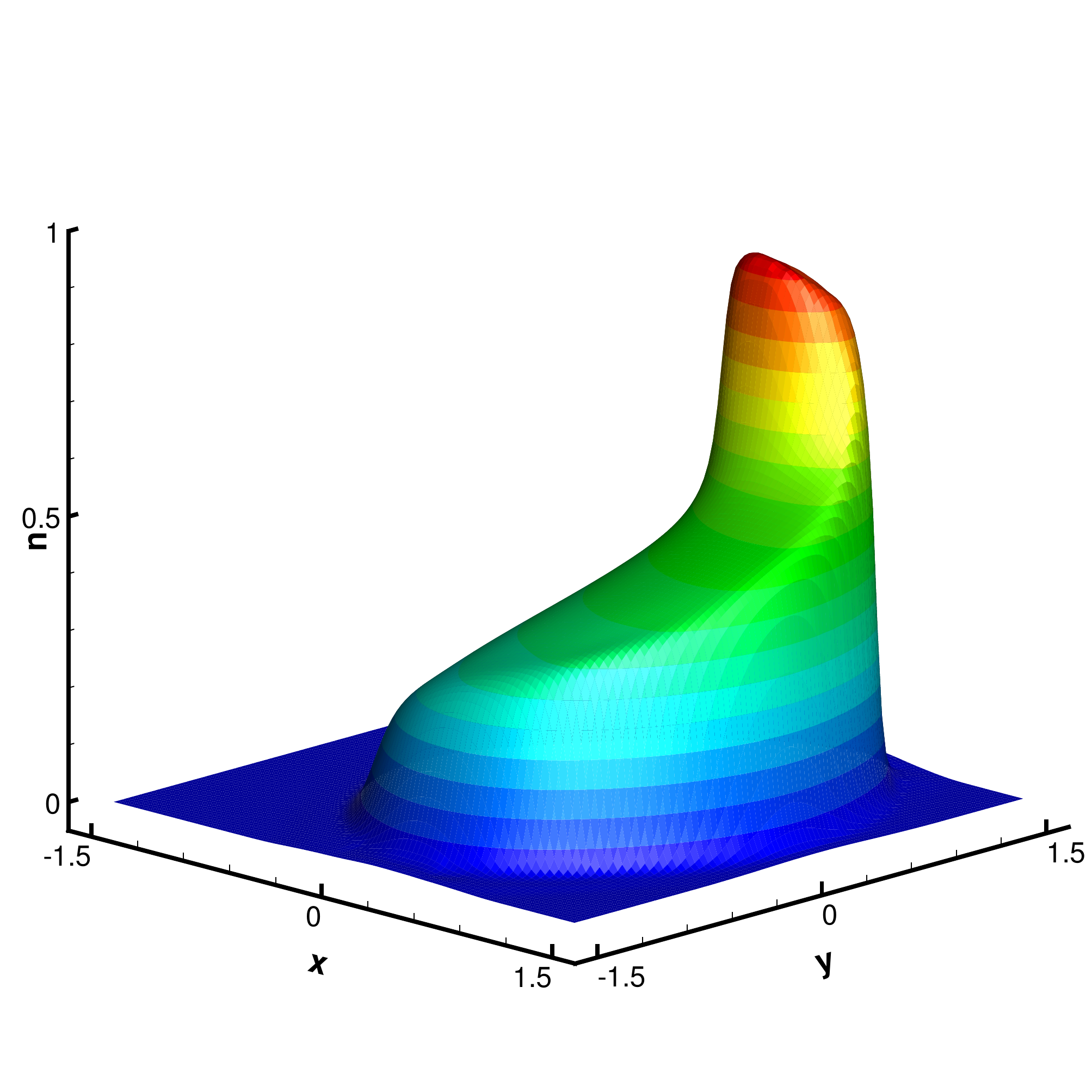}\label{Fig13.2}}
	\caption{\em Example 7: 2D Buckley-Leverett Equation. $k=3$. $T=0.5$. $CFL=0.5$. $200\times200$ grid points.}
	\label{Fig13}
\end{figure} 
\section{Conclusion}

In this paper, we proposed a novel numerical scheme to solve the nonlinear degenerate parabolic equations with non-smooth solutions. In such a framework, the spatial derivatives were represented as a special kernel based formulation of the solutions found in the method of lines transpose framework, and a fast summation algorithm was used to reduce the computational complexity of the kernel based approach to $O(N)$.  The kernel based formulation used in this work is as fast as explicit time stepping methods.  In time, we coupled the scheme with the high order explicit SSP RK method. Theoretical investigations indicated that the proposed scheme is unconditionally stable up to third order accuracy. Therefore, the new method allowed for much larger time step evolution compared with other explicit schemes with the same order accuracy. Moreover, to avoid spurious oscillations, a high order WENO methodology and a nonlinear filter are further employed. A collection of numerical tests verified the performance of the proposed scheme, demonstrating both its designed high order accuracy and the ability to produce non-oscillatory shock transitions for discontinuous solutions. 
Future work consists of extending the scheme to other equations and deal with genenral boundary conditions.

\appendix
\section{Proof of Lemma \ref{lem1}}

	Here, we only give the proof for the case of $\mathcal{D}_{0}$. For $\mathcal{D}_{L}$ and $\mathcal{D}_{R}$, the proof can be established by a similar idea.
		
	Using the definition of $I^{0}$ and integration by parts twice, we have
	{\small
	\begin{align}
	\label{eq:recur}
	I^{0}[v,\alpha](x)
	=v(x) + \frac{1}{\alpha^2}I^{0}[v_{xx},\alpha](x) -  \left(\frac{1}{2}v(a)-\frac{1}{2\alpha}v_{x}(a)\right) e^{-\alpha(x-a)} -\left(\frac{1}{2}v(b)+\frac{1}{2\alpha}v_{x}(b)\right)e^{-\alpha(b-x)}.
	\end{align}}
	Thus,
	\begin{small} 
	\begin{align*}
	\label{eq:recurD}
	\mathcal{D}_{0}[v,\alpha](x)
	=  -\frac{1}{\alpha^{2}}I^{0}[v_{xx},\alpha](x) -\left(A_{0}[v,\alpha]-\frac{1}{2}v(a)+\frac{1}{2\alpha}v_{x}(a)\right) e^{-\alpha(x-a)} 
	-\left(B_{0}[v,\alpha]-\frac{1}{2}v(b)-\frac{1}{2\alpha}v_{x}(b)\right) e^{-\alpha(b-x)}.
	\end{align*}
	\end{small}
	Here, $A_{0}[v,\alpha]$ and $B_{0}[v,\alpha]$ are obtained from the boundary treatment of $\mathcal{D}_{0}[v,\alpha]$ \eqref{eq:bcD0per}. Moreover, based on \eqref{eq:recur}, $A_{0}[v,\alpha]$ and $B_{0}[v,\alpha]$ can be rewritten as 
	\begin{align*}
	A_{0}[v,\alpha]=& \frac{1}{1-\mu} \left( \frac{1}{\alpha^2} I^{0}[v_{xx},\alpha](b) 
		-  \frac{1}{2} \left( v(a)-\frac{1}{\alpha}v_{x}(a) \right) \mu + \frac{1}{2}\left( v(b)-\frac{1}{\alpha}v_{x}(b) \right) \right),\\
	B_{0}[v,\alpha]=& \frac{1}{1-\mu} \left( \frac{1}{\alpha^2} I^{0}[v_{xx},\alpha](a) +  \frac{1}{2}\left( v(a)+\frac{1}{\alpha}v_{x}(a) \right) - \frac{1}{2}\left( v(b)+\frac{1}{\alpha}v_{x}(b) \right) \mu  \right).
	\end{align*}
	Therefore, we have
	\begin{align*}
		\mathcal{D}_{0}[v,\alpha](x)
		=&  -\frac{1}{\alpha^{2}}I^{0}[v_{xx},\alpha](x) -\frac{1}{\alpha^2}\frac{I^{0}[v_{xx},\alpha](b)}{1-\mu}e^{-\alpha(x-a)}
		-\frac{1}{\alpha^2}\frac{I^{0}[v_{xx},\alpha](a)}{1-\mu} e^{\alpha(b-x)}\\
		=& -\frac{1}{\alpha^2} \mathcal{L}_{0}^{-1}[v_{xx},\alpha](x)
		= -\frac{1}{\alpha^2} v_{xx}(x) +	\frac{1}{\alpha^2}\mathcal{D}_{0}[v_{xx},\alpha](x)
	\end{align*}
	Upon iterating this process $k$ times, we obtain that
	\begin{align*}
			\mathcal{D}_{0}[v,\alpha](x)
			=&  -\sum_{p=1}^{k}\frac{1}{\alpha^{2p}}\partial^{2p}_{x}v(x) -\frac{1}{\alpha^{2k+2}}\mathcal{L}^{-1}_{0}[\partial^{2k+2}_{x}v,\alpha](x).
	\end{align*}
			
\section{Formulation of WENO quadrature}

Here, we list the coefficients in WENO quadrature  approximating $J^{L}_{i}$ with six points:
\begin{align*}
c^{(0)}_{-3} =& \frac{6-6\nu+2\nu^2-(6-\nu^2 )e^{-\nu}}{6\nu^3}, \\
c^{(0)}_{-2} =& -\frac{6-8\nu+3\nu^2-(6-2\nu-2\nu^2)e^{-\nu}}{2\nu^3}, \\
c^{(0)}_{-1} =& \frac{6-10\nu+6\nu^2-(6-4\nu-\nu^2+2\nu^3)e^{-\nu}}{2\nu^3},\\
c^{(0)}_{0} =& -\frac{6-12\nu+11\nu^2-6\nu^3-(6-6\nu+2\nu^2)e^{-\nu}}{6\nu^3},\\
c^{(1)}_{-2} =& \frac{6-\nu^2-(6+6\nu+2\nu^2)e^{-\nu}}{6\nu^3},\\
c^{(1)}_{-1} =&-\frac{6-2\nu-2\nu^2-(6+4\nu-\nu^2-2\nu^3)e^{-nu}}{2\nu^3}, \\
c^{(1)}_{0} =& \frac{6-4\nu-\nu^2+2\nu^3-(6+2\nu-2\nu^2)e^{-\nu}}{2\nu^3}, \\
c^{(1)}_{1}=& -\frac{6-6\nu+2\nu^2-(6-\nu^2)e^{-\nu}}{6\nu^3}, \\
c^{(2)}_{-1} =& \frac{6+6\nu+2\nu^2-(6+12\nu+11\nu^2+6\nu^3)e^{-\nu}}{6\nu^3}, \\
c^{(2)}_{0} =& -\frac{6+4\nu-\nu^2-2\nu^3-(6+10\nu+6\nu^2)e^{-\nu}}{2\nu^3}, \\
c^{(2)}_{1}=& \frac{6+2\nu-2\nu^2-(6+8\nu+3\nu^2)e^{-\nu}}{2\nu^3},\\
c^{(2)}_{2}=& -\frac{6-\nu^2-(6+6\nu+2\nu^2)e^{-\nu}}{6\nu^3} u_{i+2}.
\end{align*}
And the linear weights are
\begin{align*}
& d_{0} =\frac{6-\nu^2-(6 + 6\nu + 2\nu^2 )e^{-\nu} }{3\nu ((2-\nu)-(2+\nu)e^{-\nu})}\\
& d_{2}=\frac{ 60-60\nu+15\nu^2+5\nu^3-3\nu^4-(60-15\nu^2+2\nu^4)e^{-\nu}}{10\nu^2(6-\nu^2-(6+6\nu+2\nu^2)e^{-\nu})}\\
& d_{1}=1-d_{0}-d_{2}
\end{align*}

\bibliographystyle{abbrv}
\bibliography{ref}

\end{document}